\documentclass[reqno, 12pt]{amsart}

\usepackage{amsfonts, amsthm, amsmath, amssymb}
\usepackage{hyperref}
\hypersetup{colorlinks=false}
\usepackage[all]{xy}

\usepackage[margin=1.5in]{geometry}

\usepackage{helvet}

\RequirePackage{mathrsfs} \let\mathcal\mathscr

\numberwithin{equation}{section}

\newtheorem{theorem}{Theorem}[section]
\newtheorem{lemma}[theorem]{Lemma}

\newtheorem{corollary}[theorem]{Corollary}

\theoremstyle{definition}
\newtheorem*{ack}{Acknowledgements}

\renewcommand{\d}{\mathrm{d}}
\renewcommand{\phi}{\varphi}
\renewcommand{\rho}{\varrho}

\newcommand{\sumstar}{\sideset{}{^*}\sum}

\newcommand{\card}{\#}
\newcommand{\0}{\mathbf{0}}

\newcommand{\mmod}[1]{\,\,\text{mod}\,\,#1}

\newcommand{\PP}{\mathbb{P}}
\renewcommand{\AA}{\mathbb{A}}
\newcommand{\A}{\mathbf{A}}
\newcommand{\FF}{\mathbb{F}}
\newcommand{\ZZ}{\mathbb{Z}}

\newcommand{\NN}{\mathbb{N}}
\newcommand{\QQ}{\mathbb{Q}}
\newcommand{\RR}{\mathbb{R}}

\newcommand{\N}{\mathbb{N}}
\newcommand{\Q}{\mathbb{Q}}
\newcommand{\Z}{\mathbb{Z}}
\newcommand{\F}{\mathbb{F}}

\newcommand{\cE}{\mathcal{E}}
\newcommand{\cS}{\mathcal{S}}

\newcommand{\cW}{\mathcal{W}}
\newcommand{\cR}{\mathcal{R}}

\newcommand{\cU}{\mathcal{U}}
\newcommand{\cV}{\mathcal{V}}
\newcommand{\cN}{\mathcal{N}}
\newcommand{\cZ}{\mathcal{Z}}

\newcommand{\calE}{\mathcal{E}}

\newcommand{\calM}{\mathcal{M}}
\newcommand{\calN}{\mathcal{N}}

\newcommand{\calR}{\mathcal{R}}

\newcommand{\calV}{\mathcal{V}}
\newcommand{\calW}{\mathcal{W}}

\newcommand{\val}{{\rm val}}

\newcommand{\XX}{\boldsymbol{X}}

\renewcommand{\leq}{\leqslant}

\renewcommand{\geq}{\geqslant}
\renewcommand{\ge}{\geqslant}
\renewcommand{\bar}{\overline}

\newcommand{\x}{\mathbf{x}}
\newcommand{\y}{\mathbf{y}}

\renewcommand{\v}{\mathbf{v}}
\renewcommand{\u}{\mathbf{u}}

\newcommand{\w}{\mathbf{w}}

\newcommand{\p}{\mathbf{p}}
\newcommand{\q}{\mathbf{q}}
\newcommand{\s}{\mathbf{s}}

\newcommand{\bfx}{\mathbf{x}}

\newcommand{\fo}{\mathfrak{o}}

\newcommand{\fF}{\mathfrak{F}}

\newcommand{\grS}{\mathfrak{S}}

\newcommand{\la}{\lambda}

\newcommand{\al}{\alpha}
\newcommand{\del}{\delta}
\newcommand{\Del}{\Delta}
\newcommand{\alp}{\alpha}
\newcommand{\bet}{\beta}

\newcommand{\kap}{\kappa}
\newcommand{\lam}{\lambda}
\newcommand{\ome}{\omega}
\newcommand{\sig}{\sigma}
\newcommand{\vartet}{\vartheta}

\newcommand{\ve}{\varepsilon}

\DeclareMathOperator{\rank}{rank}

\DeclareMathOperator{\tr}{Tr}
\DeclareMathOperator{\Tr}{Tr}
\DeclareMathOperator{\nf}{\mathbf{N}}

\DeclareMathOperator{\Br}{Br}

\renewcommand{\hat}{\widehat}

\newcommand{\str}{\widetilde{{\tr}}}

\newcommand{\Lring}{\mathfrak{o}_L}

\newcommand{\alpmu}{\alp_{m,\u^{(m)}} (x)}
\newcommand{\alpmuhat}{\hat{\alp}_{m,\u^{(m)}} (x)}
\newcommand{\alpo}{\alp^\dagger_{m,\u^{(m)}} (x)}
\newcommand{\betmu}{\bet_{m,\v^{(m)}}(y)}
\newcommand{\omemu}{\omega_m}
\newcommand{\na}{n_1}
\newcommand{\nb}{n_2}
\newcommand{\Ka}{K_1}
\newcommand{\Kb}{K_2}
\newcommand{\Na}{N_1}
\newcommand{\Nb}{N_2}
\newcommand{\Sn}{Sv_0}

\title[Strong approximation]{Strong approximation and a conjecture of Harpaz and Wittenberg}

\author{T.D. Browning}
\address{School of Mathematics\\
University of Bristol\\ Bristol\\ BS8 1TW\\ UK}
\email{t.d.browning@bristol.ac.uk}

\author{D. Schindler} 
\address{Utrecht University\\Hans Freudenthalgebouw\\Budapestlaan 6\\3584 CD Utrecht\\
Netherlands}
\email{d.schindler@uu.nl}

\date{\today}

\begin{document}

\begin{abstract}
We study strong approximation for some algebraic varieties over $\QQ$ which are defined using norm forms. This allows us to confirm a special case of a conjecture due to Harpaz and Wittenberg.
\end{abstract}

\maketitle

\setcounter{tocdepth}{1}
\tableofcontents

\section{Introduction}

This  paper establishes a special case of a conjecture due to 
Harpaz and Wittenberg \cite[Conj.~9.1]{HW}, the resolution of which leads to the following very general result  about the behaviour of rational points on varieties  over $\QQ$  admitting a suitable  fibration.

\begin{theorem}\label{thm:HW} 
Let $X$ be a smooth proper and geometrically irreducible
variety over $\QQ$,  and let $f : X \to \PP_\QQ^1$
be a dominant morphism with rationally connected
geometric generic fibre.  Suppose that $\rank(f)\leq 3$, 
with at least one non-split fibre lying over a rational point of $\PP_\QQ^1$.
Assume that there exists a  Hilbert subset $H\subset \PP^1_\QQ$ such that 
 $X_c(\QQ)$ is dense in $X_c(\A_\QQ)^{\Br(X_c)}$ for every rational point $c$ in $H$. 
Then $X(\QQ)$ is dense in $X(\A_\QQ)^{\Br(X)}$.
\end{theorem}

Recall here that the rank of a fibration $f:X\to \PP_\QQ^1$ is defined to be 
the sum of the degrees of the closed points of $\PP_\QQ^1$ above which the fibre of $f$ is not split.
The definition of split fibres can be found in work of  Skorobogatov \cite{ajm}, which is where the notion was originally  introduced to the subject.
The conclusion of Theorem \ref{thm:HW}  is also true when  $\rank(f)\leq 2$ over any number field, 
without any condition on the non-split fibres (see  \cite[Thm.~9.31]{HW} and its footnote). The latter is due to Harari \cite{harari} when $\rank(f)=1$.
When $\rank(f)=2$ and the fibres satisfy weak approximation it follows from 
work of   Colliot-Th\'el\`ene and Skorobogatov \cite{ct-sk}.
Thanks to 
Matthiesen \cite{M} and the work Harpaz and Wittenberg
 \cite[Thm.~9.28]{HW},  
the result is also known to be true for arbitrary values of $\rank(f)$ over $\QQ$, provided that all the non-split fibres lie over rational points of $\PP_\QQ^1$.

Next, 
let $K/\QQ$ be a finite  extension of number fields of degree $n\geq 2$ 
and fix a $\QQ$-basis $\{\omega_1,\dots,\omega_n\}$ for $K$ 
over $\QQ$. 
For any subfield $F\subset K$,  we denote by 
$$
\nf_{K/F}(x_1,\dots,x_n)=
N_{K/F}(x_1\omega_1+\dots+x_n\omega_n)
$$ 
the corresponding norm form, where $N_{K/F}$ is the field norm.
It follows from work of Derenthal, Smeets and Wei \cite[Thm.~2]{DSW}
that the Brauer--Manin obstruction to the Hasse principle or weak approximation is the only obstruction on any smooth proper model $X$ of the affine variety 
\begin{equation} \label{PNZ'}
P(t)=\nf_{K/\QQ}(x_1,\ldots,x_n),
\end{equation}
where  $P(t)$ is an irreducible  quadratic polynomial over $\QQ$.  
The obvious morphism $X\to \PP_\QQ^1$  has rational geometric  generic fibre. 
It has precisely two non-split fibres over $\PP_\QQ^1$, one of which is the fibre at infinity 
and the other lies above the quadratic point defined by $P(t)$. 
Moreover the smooth fibres  over $\PP^1_\QQ(\QQ)$  
all satisfy the property that the Brauer--Manin obstruction is the only obstruction to the Hasse principle or weak approximation by work of Sansuc \cite{sansuc}.
Hence  Theorem \ref{thm:HW} applies to $X$ and may be viewed as a considerable generalisation of \cite[Thm.~2]{DSW}.
For example, it evidently  applies to smooth proper models of the affine varieties  in which the right hand side 
of \eqref{PNZ'}
is replaced by a product of norm forms.

Theorem \ref{thm:HW} will follow from the study of strong approximation for a particular family of varieties defined using norm forms. 
For any algebraic variety $Y$ defined over $\QQ$, we say that 
strong approximation holds for $Y$ off a finite set $S$ of places of $\QQ$ if the image of $Y(\QQ)$ is dense in the space $Y(\A_\QQ^S)$ of ad\`elic points outside $S$.  
%In particular with this definition strong approximation off $S$ holds if  $Y(\A_\QQ)=\emptyset$.
%This is not true!
%It could be that Y(A_Q) is empty while Y(A_Q^S) is not, in which case Y(Q)
%will fail to be dense in Y(A_Q^S).  One solution would be to take the
%convention that strong approximation off S holds whenever Y fails to have
%local points at the places of S.
We will follow the convention that strong 
approximation off $S$ holds whenever $Y$ fails to have local points at the places of $S$.
Studying strong approximation  and the integral Hasse principle on integral models of  affine varieties is generally harder than studying weak approximation and the Hasse principle for rational points on proper models of $Y$.

It is now time to introduce the auxiliary variety $W$ whose arithmetic lies at the heart of  Theorem \ref{thm:HW}.
For $i\in\{1,2\}$,
 let $K_i/\QQ$  be an arbitrary number field of degree $n_i$.
Let $L=\QQ(\sqrt{a})$ 
for any   $a\in \QQ^*\setminus {\QQ^*}^2$.  We henceforth assume that $L\subset K_1$. In particular  $n_1$ is even.
Let $\delta\in L^*$ 
and let  $V\subset \AA_\QQ^{n_1+n_2}$ be the variety given by the equation
\begin{equation}\label{eq:1}
\tr_{L/\QQ}\big(\delta \nf_{K_1/L}(\y)\big)=2\nf_{K_2/\QQ}(\w).
\end{equation}
Let $Z\subset V$ be the codimension two subvariety in which either
$
\nf_{K_1/\QQ}(\y)= \nf_{K_2/\QQ}(\w)=0
$ 
or, if 
one factors $\nf_{K_1/\QQ}(\y)$
(resp.~ $\nf_{K_2/\QQ}(\w)$) over $\bar \QQ$ as a product of linear forms, then two or more 
of the factors vanish at $\y$ (resp.~ at $\w$). 
The auxiliary variety in which   we are interested   is defined to be the  open subset 
$W=V\setminus Z$. We shall prove the following result.

\begin{theorem}\label{thm:1}
Strong approximation holds for $W$ off any non-empty finite set of places. 
\end{theorem}

Consider momentarily the special case $K_1=K_2$. Then 
the variety $V$ in \eqref{eq:1} first arose in work of Browning and Heath-Brown \cite{BHB} in their pioneering  investigation of the Hasse principle and weak approximation for \eqref{PNZ'}.
Our variety    $W$ is a smooth open subset of $V$ and it contains the variety \cite[Eq.~(1.8)]{BHB} as a dense open subset.  In particular it follows from \cite[Thm.~2]{BHB} that $W$  satisfies the Hasse principle and weak approximation when $K_1=K_2$. 
Although our strategy to prove Theorem \ref{thm:1} is inspired by the work of Browning and Heath-Brown \cite{BHB}, it involves a number of new difficulties and ideas on the analytic side.
While there is little difficulty in  handling  $K_1\neq K_2$ in 
\eqref{eq:1}, one serious obstacle arises  from an extra ``square-freeness" condition that occurs  when  dealing with strong approximation for the particular open set $W\subset V$. In a different direction one is faced with the additional challenge of a sum defining the singular series  that is not a priori absolutely convergent. The process of completing and then interpreting the singular series is only achieved through a delicate analysis of local counting functions. 

The deduction of Theorem \ref{thm:HW} from Theorem \ref{thm:1}  is carried out in \S \ref{s:deduce}. 
The remaining sections are concerned with the proof of Theorem \ref{thm:1}, beginning with \S \ref{sec2}, where we relate the statement of the theorem to a suitable counting problem.
In \S \ref{s:4} we lay down the necessary tools to handle the square-freeness condition that occurs  in our work, and for which we build on the work of Matthiesen \cite{M} mentioned above. 
Finally, in \S \ref{s:6}  and \S \ref{s:7} the main term of our counting function is analysed and shown to satisfy the properties required for the conclusion of Theorem \ref{thm:1}.

\begin{ack}
The authors are very grateful to Olivier Wittenberg for numerous helpful comments and for setting this project in motion by asking for a proof of  Theorem \ref{thm:1}.
While working on this paper the 
first author was supported by {\em ERC grant} \texttt{306457} and the second author by the {\em NSF} under agreement No. \texttt{DMS-1128155}.
The authors are also grateful to the anonymous referee for numerous helpful comments.
\end{ack} 

\section{Theorem \ref{thm:1} implies Theorem \ref{thm:HW}}\label{s:deduce}

Our starting point is the  construction of $W$ outlined at the start of \cite[\S 9.2.2]{HW},
with data $n=2$, $k=k_1=\QQ$ and  $k_2$ a quadratic extension of $\QQ$.
On carrying out a non-singular linear change of variables on $(\lambda,\mu)$, one may clearly assume that   $a_1=0$ and  that $a_2$ is the square root of a rational
number. In this way we arrive at \eqref{eq:1} with $\delta=b_1b_2^{-1}$.
But then it is straightforward to confirm that the following result is a direct consequence of Theorem~\ref{thm:1} and \cite[Cor.~9.10]{HW}.

\begin{corollary}\label{cor:HW}
Conjecture 9.1 in \cite{HW} holds when $n=2$, $k=k_1=\QQ$ and $k_2$ is a quadratic extension of $\QQ$.
\end{corollary}

We may now prove Theorem \ref{thm:HW}.  
Let  $f : X \to \PP_\QQ^1$ be as in the statement of the theorem. 
In particular the hypotheses (1), (2) and (4) of \cite[Cor.~9.23]{HW} are met (cf. the proof of \cite[Cor.~9.25]{HW}).
After a change of coordinates we may assume that the fibre $f^{-1}(\infty)$ is split.
Let $M\subset \AA_\QQ^1$ be the finite closed subset containing the points that have a  non-split fibre. Then the hypotheses of Theorem \ref{thm:HW} imply that $M$ contains at least one rational point. If all of the 
points of $M$ are rational then the result is a special case of a theorem of  Matthiesen \cite{M}, as recorded in \cite[Thm.~9.28]{HW}. Alternatively, we may suppose that $M$ consists of a rational point and a point defined over a quadratic extension of $\QQ$.
In this case the statement of Theorem \ref{thm:HW} is found to be  a straightforward consequence of Corollary \ref{cor:HW} and 
\cite[Cor.~9.24]{HW}.

\section{From strong approximation to counting}\label{sec2}

Let $W\subset \AA_\QQ^{n_1+n_2}$ be the open subset of the variety $V$ in \eqref{eq:1}, as defined in the introduction.  
According to  \cite[Prop.~2.2]{ct-xu}, 
in order to prove Theorem~\ref{thm:1} it will suffice to show that the variety $W$ satisfies strong approximation off an arbitrary  place 
$v_0$ of $\QQ$.
Let $\Omega$ denote the set of places of $\QQ$ and write $\x=(\y,\w)$ for the vector of variables appearing in the definition of $W$.  
We fix an integral  model $\cV$ for $V$, which is obtained by clearing denominators from 
\eqref{eq:1}. 
This leads to an equation 
\begin{equation}\label{eq:1'}
\tr_{L/\QQ}\big(\delta \nf_{K_1/L}(\y)\big)=c\nf_{K_2/\QQ}(\w),
\end{equation}
with $a\in \ZZ$ square-free,  $\delta\in \fo_L\setminus \{0\}$ 
and $c\in 2\ZZ\setminus \{0\}$. Moreover, we may assume that 
for $i\in \{1,2\}$ the norm forms are defined using 
a $\ZZ$-basis 
  $\{\omega_1^{(i)},\dots,\omega_{n_i}^{(i)}\}$
for $\fo_{K_i}$,    with $\omega_1^{(i)}=1$.
We let $\cZ$ be the scheme-theoretic closure of $Z$ in $\cV$ and put $\cW=\cV\setminus \cZ$.

 For strong approximation off $v_0$ on $W$ we must show the following: for any finite set of places $S\subset \Omega\setminus \{v_0\}$, any $(\x_v)\in W(\A_\QQ)$ with $\x_v\in \cW(\ZZ_v)$ for all $v\not\in S\cup \{v_0\}$, there exists a point 
$\x\in W(\QQ)$ with $\x\in \cW(\ZZ_{v})$ for all $v\not\in S\cup \{v_0\}$, such that $\x$ is arbitrarily close to $\x_v$ for all $v\in S$.  
Rather than asking that $\x\in \cW(\ZZ_v)$, for all 
$v\not\in S\cup \{v_0\}$,  we shall demand that 
$\x\in \cW^\circ(\ZZ_v)$ for all 
$v\not\in S\cup \{v_0\}$,
where $\cW^\circ(\ZZ_v)$ is the set of  points $\x=(\y,\w)\in \cV(\ZZ_v)$ 
%Here, for any finite  place $v\in \Omega$, 
%elements of 
%$\cW^\circ(\ZZ_v)$ correspond to points $\x=(\y,\w)\in \cV(\ZZ_v)$ 
such that 
\begin{equation}\label{eq:sq-free}
\begin{split}
\min\{\val_v(\nf_{K_1/\QQ}(\y)), \val_v(\nf_{K_2/\QQ}(\w))\}&=0,\\
\max\{\val_v(\nf_{K_1/\QQ}(\y)), \val_v(\nf_{K_2/\QQ}(\w))\}&\leq 1.
\end{split}
\end{equation}
This is clearly stronger than is strictly necessary,  but as  it turns out,  it is easier to handle  within the confines of our analytic  arguments.

Let $S\subset \Omega\setminus \{v_0\}$ be a finite set  of places and  let $(\x_v)\in W(\A_\QQ)$, with $\x_v\in \cW(\ZZ_v)$ for all $v\not\in S\cup \{\infty, v_0\}$.
It will be convenient to put $S_f=S\setminus \{\infty\}$ for the set of finite places in $S$. 
There are now two cases to consider, depending on whether or not $v_0$ is a finite place.  Suppose first that $v_0=\infty$. Then $S=S_f$ and 
we must find 
$\x\in W(\QQ)$ with $\x\in \cW^\circ(\ZZ_{v})$, for all $v\not\in S_f\cup \{\infty\}$, such that $\x$ is arbitrarily close to $\x_v$ for all $v\in S_f$.
Alternatively, suppose  that $v_0$ is a finite place. 
Without loss of generality  we may assume 
that $S$  contains the infinite place and we put $S_f=S\setminus \{\infty\}$ as before.
In this case we must find 
$\x\in W(\QQ)$ with $\x\in \cW^\circ(\ZZ_{v})$, for all $v\not\in S\cup \{v_0\}$, such that $\x$ is arbitrarily close to $\x_v$, for all $v\in S=S_f\cup\{\infty\}$.
Thus, when $v_0=\infty$ we only have to approximate at a finite collection of finite local places $S_f$, but when $v_0$ is finite we also have to approximate at the real place.

Let $C\in \ZZ$ with $C^{-1}\in \ZZ_{S_f}$ (i.e. all prime factors of $C$ lie in $S_f$) be chosen so that 
$\x_v'=(C^{2n_2}\y_v,C^{n_1}\w_v)\in \ZZ_v^{n_1+n_2}$ for all $v\in S_f$. The change of 
variables that replaces $\x=(\y,\w)$ by $(C^{2n_2}\y,C^{n_1}\w)$ clearly maps $(\x_v)\in W(\A_\QQ)$ to 
$(\x_v')\in W(\A_\QQ)$,  with $\x_v'\in \cW(\ZZ_v)$ for all 
$v\in \Omega\setminus \{v_0,\infty\}$.
By the Chinese remainder theorem we can find $\x^{(M)}\in\ZZ^{n_1+n_2}$ arbitrarily close to $\x_v'$ for all   $v\in S_f$.
We now seek a point $\x'=(\y',\w')\in W(\QQ)\cap \cV(\ZZ)$ 
which is very close to $\x^{(M)}$ in the $v$-adic topology 
for all $v\in S_f$. We further require that $\x'\in \cW^\circ(\ZZ_v)$ for all $v\not\in S\cup\{\infty,v_0\}$. The first condition  translates into the conditions
$$
\y'\equiv \y^{(M)} \bmod{M}, \quad 
\w'\equiv \w^{(M)} \bmod{M},
$$
for a suitable positive integer $M$ built from the primes in $S_f$. The second condition \eqref{eq:sq-free}
can be written
$$
\mu_S^2\left(\nf_{K_1/\QQ}(\y')\nf_{K_2/\QQ}(\w')\right)=1,
$$
where 
for any integer $k$
we set
\begin{equation}\label{eq:heron}
\mu_S^2(k)
:=
\begin{cases}
1 &\mbox{if $p^2\nmid k$ for all primes $p\not\in S\cup \{v_0\}$, }\\
0 &\mbox{otherwise}.
\end{cases}
\end{equation}

Once we've found such a vector $\x'$, then this is also very close to $\x_v'\in \cV(\ZZ_v)$ for all $v\in S_f$, and then 
$$
\x=(C^{-{2n_2}}\y',C^{-{n_1}}\w')\in W(\QQ)\cap \cW(\ZZ_{S_f\cup\{v_0\}})
$$
is very close to  $\x_v$ for all $v\in S_f$, with 
$\x\in \cW^\circ(\ZZ_v)$ for all $v\not\in S\cup\{v_0\}$.
This will completely answer strong approximation off the infinite place. 
Let $\ve>0$ be arbitrary. 
Whether or not  $v_0$ is a finite place,   we will further demand that our  rational point 
$\x'\in W(\QQ)\cap \cV(\ZZ)$ satisfies
$$
|\y'/Y-\y_\infty|<\ve , \quad  |\w'/W-\w_\infty|<\ve , 
$$
for suitable parameters $Y, W>0$.  
Define 
$$
 P(k):=\{k^j: j\geq 0\}
$$ 
for the set of  powers of a positive integer $k$.  
If  $v_0$ is a finite place and 
 $p_0$ is the prime corresponding to $v_0$ then we will  take
$Y,W\in P(p_0^{\phi(M)})$ such that $Y^{n_1}=W^{2n_2}$. 
Then it is clear that $\x''=(\y'/Y,\w'/W)$ will satisfy the constraints required to deduce strong approximation off $v_0$. 
If  $v_0=\infty$ then the vector $\x'$ is sufficient to deduce strong approximation off infinity and 
we may allow $Y,W$ to be arbitrary positive real numbers such that 
$Y^{n_1}=W^{2n_2}$.

To summarise our argument so far, let $v_0$ be a fixed place of $\QQ$ and let $S\subset \Omega\setminus \{v_0\}$ be a finite set  of places.
There is no loss of generality in assuming that $S\cup \{v_0\}$ contains the primes dividing $ac N_{L/\QQ}(\delta)$, 
together with the primes which ramify in $K_1$ or $K_2$. We may now draw the following conclusion.

\begin{lemma}\label{lem:find-yw} 
Assume that $(\x_v)\in W(\A_\QQ)$
Then there exists  
$M\in \NN$ and 
a solution $(\y^{(M)},\w^{(M)})$ of \eqref{eq:1'} 
over $\ZZ/M\ZZ$, with    $ (\y^{(M)},\w^{(M)})\equiv (\y_p,\w_p) \bmod{M}$ for any prime $p\in  S_f$,
together with a point   $(\y^{(\RR)},\w^{(\RR)})\in W(\RR)$,  having the following property.
For any $\ve>0$ and all large enough values of
$W,Y\in \NN$ such that $Y^{n_1}=W^{2n_2}$, 
suppose there is a point 
$(\y, \w)\in \cV(\ZZ)$ satisfying 
\begin{align*}
\mu_S^2\left(\nf_{K_1/\QQ}(\y)\nf_{K_2/\QQ}(\w)\right)=1,\\
\y\equiv \y^{(M)} \bmod{M}, \quad 
\w\equiv \w^{(M)} \bmod{M},
\end{align*}
and
\begin{equation}\label{eq:size2}
\begin{split}
|\y-Y\y^{(\RR)}|<\ve Y, \quad 
|\w-W\w^{(\RR)}|<\ve W.
\end{split}
\end{equation}
Then there exists a point 
$\x\in W(\QQ)$ with $\x\in \cW(\ZZ_{v})$ for all $v\not\in S\cup \{v_0\}$, such that $\x$ is arbitrarily close to $\x_v$ for all $v\in S$.  
\end{lemma}

This result shows that in order to prove that 
 $W$ satisfies strong approximation off $v_0$ it suffices find $(\y,\w)\in \cV(\ZZ)$ satisfying the constraints of the lemma.

Since $(\y^{(\RR)},\w^{(\RR)})\in W(\RR)$ we must have 
$\nabla\nf_{K_1/\QQ}(\y^{(\RR)})\neq \0$. 
In particular $\y^{(\RR)}\neq \0$.
For technical reasons it will be convenient to move $\w^{(\RR)}$ in \eqref{eq:1'} very slightly, and make a corresponding adjustment in $\y^{(\RR)}$ to compensate, in order to ensure that we also have  $\w^{(\RR)}\neq \0$.
In particular, on choosing
$\ve$ sufficiently small, we can make sure that  that $(\y,\w)\neq (\0,\0)$  when \eqref{eq:size2} holds.

We introduce additional bilinear structure by working  
with a  variety of higher dimension. 
We'll  proceed by searching for suitably localised solutions 
$$
(\u, \v, \w)\in \ZZ^{n_1}\times\ZZ^{n_1}\times\ZZ^{n_2}
$$ 
to the Diophantine equation
\begin{equation}\label{eq:UT} 
\tr_{L/\QQ}\big(\delta \nf_{K_1/L}(\u)\nf_{K_1/L}(\v)\big)=
c\nf_{K_2/\QQ}(\w).
\end{equation}
Let 
$$
\u^{(M)}=\y^{(M)}, \quad 
\u^{(\infty)}=\y^{(\RR)} \quad \text{ and }
\quad \v^{(M)}=\v^{(\RR)}=(1,0,\dots,0).
$$
Let $W,U,V\in P(p_0^{\phi(M)})$ such that $(UV)^{n_1}=W^{2n_2}$.
We suppose that we have a solution $(\u,\v,\w)\in \ZZ^{2n_1+n_2}$ of \eqref{eq:UT} which satisfies 
\begin{equation}\label{eq:sq-1}
\mu_S^2\left(\nf_{K_1/\QQ}(\u)\nf_{K_1/\QQ}(\v)\nf_{K_2/\QQ}(\w)\right)=1
\end{equation}
and 
$$
\u\equiv \u^{(M)} \bmod{M}, \quad 
\v\equiv \v^{(M)} \bmod{M}, \quad 
\w\equiv \w^{(M)} \bmod{M},
$$
and
$$
|\u-U\u^{(\RR)}|<\ve U, \quad 
|\v-V\v^{(\RR)}|<\ve V, \quad 
|\w-W\w^{(\RR)}|<\ve W.
$$
Then if $\y\in \ZZ^{n_1}$ is the vector corresponding to 
 $(\sum_{i} u_i \omega_i^{(1)})
(\sum_{i} v_i \omega_i^{(1)})$, when multiplied out and expressed in terms of the integral basis,
 it is easy to check that the vector $(\y,\w)\in \ZZ^{n_1+n_2}$ will be a solution of \eqref{eq:1'} satisfying the conditions of Lemma 
\ref{lem:find-yw} with $Y=UV$.

We employ the notation of 
``skew-trace''  that was introduced in \cite{BHB}.
Thus, if $\sigma$ is the non-trivial automorphism of $L$
and  $\{1,\tau\}$ is a $\ZZ$-basis for $\fo_L$, we put
$
\str(x,y):=\tr_{L/\QQ}(xy^\sigma D_L^{-1}),
$
for $x,y\in L$, where $D_L:=\tau-\tau^{\sigma}$.
%Let $\sigma$ denote the non-trivial automorphism of $L$
%and suppose that $\{1,\tau\}$ is a $\ZZ$-basis for $\fo_L$, and hence
%also a $\QQ$-basis for $L$. 
%It will be convenient to replace the
%trace $\tr_{L/\QQ}$ by a ``skew-trace'' 
%$$
%\str(x,y):=\tr_{L/\QQ}(xy^\sigma D_L^{-1})
%$$
%for $x,y\in L$, where $D_L:=\tau-\tau^{\sigma}$.
On writing $x=\delta \nf_{K_1/L}(\u)$ and 
$y=\left(\nf_{K_1/L}(\v)D_L \right)^\sigma$ 
our   equation \eqref{eq:UT} becomes
$$
\str(x,y)=c\nf_{K_2/\QQ}(\w),
$$
whereas  \eqref{eq:sq-1} becomes
$
\mu_S^2(N_{L/\QQ}(xy)\str(x,y))=1,
$
since $S$ contains all of the places which divide 
$cN_{L/\QQ}(\delta)$
or which ramify in $L$.

Next, for $i=1,2$, we let $\fF^{(i)}$ be a fundamental domain of $\fo_{K_i}/U_{K_i}^+$, where $U_{K_i}^+$ is the group of units of $\fo_{K_i}$ of norm one.  It will be convenient to set 
\begin{equation}\label{eq:domain}
\fF_X^{(i)}:=\{v\in \fF^{(i)}: |N_{K_i/\QQ}(v)|\leq X\},
\end{equation}
for any $X\geq 1$. We will abuse notation and write $\u\in \fF^{(i)}$ to mean that the point  
$u_1\omega_1^{(i)}+\dots+ u_{n_i}\omega_{n_i}^{(i)}$ belongs to the 
region $\fF^{(i)}$. Note that any vector $\x\in \RR^{n_i}\cap \fF^{(i)}$ such that $|\x|\leq X$ automatically belongs to 
$\fF_{cX^{n_i}}^{(i)}$ for an appropriate constant $c>0$ depending only on $K_i$.

We are now ready to specify the sets over which we will sum.  
Let $G$ be a further parameter, tending to
infinity with $V$. We then define the regions
\begin{align*}
\cU&:=\left\{\u \in \RR^{n_1}\cap \fF^{(1)}: |\u-U\u^{(\RR)}|<G^{-1}U
\right\},\\ 
\cV&:=\left\{\v \in \RR^{n_1}\cap \fF^{(1)}: 
|\v-V\v^{(\RR)}|<G^{-1}V
\right\},\\
\cW&:=\{\w \in \RR^{n_2}\cap \fF^{(2)}: |\w-W\w^{(\RR)}|<G^{-1}W\},
\end{align*}
where we recall that $(UV)^{n_1}=W^{2n_2}$. 
In truth the definition of $\cU$ should involve a constraint of the form
$\max_{1\leq i \leq   n_1}|L_i(\u)-UL_i(\u^{(\RR)})|<G^{-1}U$ for suitable independent linear forms 
$L_1,\dots,L_{n_1}$  (cf. \cite[Eq.~(3.14)]{BHB})
that are constructed for the sole purpose of proving  \cite[Lemma 9]{BHB}. Since we will later use the latter result as a ``black box'', we have decided to ease notation by working under the assumption that $L_i(\u)=u_i$ for $1\leq i\leq n_1$.

For technical convenience we introduce an additional restriction on the values of $\v$ that we consider. If we write $N_{K_1/L}(\v)= \Na(\v) + \Nb(\v)\tau$, then we impose the condition that
\begin{equation}\label{gcdv}
\gcd (\Na (\v),\Nb (\v))=1.
\end{equation}
Let 
\begin{align}
  \label{eq:AL}
\al(x)&:=
\#
\left\{
\u\in \cU\cap\ZZ^{n_1}: 
\begin{array}{l}
\u\equiv \u^{(M)} \bmod{M},~ \delta \nf_{K_1/L}(\u)=x
\end{array}
\right\},\\
  \label{eq:BE}
\beta(y)&:=\#\left\{
\v\in \cV\cap\ZZ^{n_1}:  
\begin{array}{l}
\v\equiv \v^{(M)} \bmod{M},~ \eqref{gcdv} \text{ holds}\\
(\nf_{K_1/L}(\v)D_L)^{\sigma}=y
\end{array}
\right\},
\end{align}
for $x,y \in \Lring$.   Lastly, we define the function
\begin{equation}
  \label{eq:LA}
  \la(l):=
\#
\left\{
\w\in \cW\cap\ZZ^{n_2}:  
\w\equiv \w^{(M)} \bmod{M},~ c\nf_{K_2/\QQ}(\w)=l  
\right\},
\end{equation}
for  $l\in \ZZ$.  Notice that $\u^{(\RR)},\v^{(\RR)},\w^{(\RR)}$ 
are all  non-zero, whence $\u,\v,\w$  will be 
non-zero throughout $\cU,\cV$ and $\cW$, if $G$ is large enough. 
It follows in particular that $\alpha, \beta$ 
are supported on non-zero $x,y\in\Lring$ and $\lambda$ is supported on non-zero integers.

We proceed to 
define the bilinear form
\begin{equation}\label{eq:msri}
\cN(G,U,V,W):=
\sum_{x,y \in \Lring }
\mu_S^2\left(N_{L/\QQ}(xy)\str(x,y)\right)
\alpha(x)\beta(y)\la\big(\str(x,y)\big).
\end{equation}
Our argument thus far  shows that  Theorem \ref{thm:1} holds if we can show that 
$\cN(G,U,V,W)>0$ for all sufficiently  large values of $G, U,V,W>0$ such that $(UV)^{n_1}=W^{2n_2}$.
In order to handle the square-freeness condition, we put 
\begin{equation}\label{eq:Sn}
\Sn=\begin{cases}
p_0\prod_{p\in S_f}p &\text{if $v_0\neq \infty$,}\\
\prod_{p\in S_f}p &\text{if $v_0= \infty$,}
\end{cases}
\end{equation}
where $p_0$ is the prime corresponding to $v_0$ when it is finite.
We may now rewrite  \eqref{eq:heron} as
\begin{equation}\label{eq:identity}
\mu_S^2(k)=\sum_{\substack{d^2\mid k\\ (d,\Sn)=1}} \mu(d),
\end{equation}
which allows us  to trade the square-freeness condition for congruence conditions. 
For small values of $d$ it is possible to adapt the argument in \cite{BHB} satisfactorily. 
The necessary tools for  handling the contribution from large values of $d$ are laid out in the next section.

\section{Handling  square-freeness}\label{s:4}

\subsection{Technical tools}
We begin with an estimate for the mean square of the functions defined in \eqref{eq:AL}--\eqref{eq:LA}.   The proof of \cite[Lemma 4]{BHB} easily  gives the following estimates. 

\begin{lemma}\label{lem:upper-abl}
Let $x\in \Lring$ and $l\in \ZZ$ be given. Then 
for any  $\eta>0$ we have
$\alpha(x)\ll_\eta U^{\eta}$, 
$\beta(x)\ll_\eta V^{\eta}$ and $\lambda(l)\ll_\eta W^{\eta}$.
\end{lemma}

It will be convenient to record the 
 mean square estimates 
$$%\begin{equation}\label{eq:mean1}
\sum_{x \in \Lring} |\alpha(x)|^2\ll_\eta U^{{n_1}+\eta},\quad
\sum_{y \in \Lring} |\beta(y)|^2\ll_\eta V^{n_1+\eta}, \quad
\sum_{l\in \ZZ}|\lambda(l)|^2 \ll_\eta W^{n_2+\eta},
$$
which are   easy consequences of Lemma \ref{lem:upper-abl}.

Let $K/\QQ$ be a number field of degree $n$. 
Let $d\in \NN$ be square-free and let $\mathfrak{F}_X$
be as in \eqref{eq:domain}, where
$\fF$ is  a fundamental domain of $\fo_{K}/U_{K}^+$.  We put $\cR(X):=\ZZ^n\cap \fF_X$ for ease of notation, and proceed to define the counting functions
\begin{equation}\label{eq:def-MN}
\begin{split}
M_d(X)&:=\#\left\{\x\in \cR(X): d^2\mid \nf_{K/\QQ}(\x)\right\},\\
N_d(\XX)&:=\#\left\{(\x,\y)\in \cR(X_1)\times \cR(X_2): d\mid \left(\nf_{K/\QQ}(\x),  \nf_{K/\QQ}(\y)\right)\right\},
\end{split}
\end{equation}
for any $d\in \NN$ and $X\in \RR_{>0}$ and $\XX=(X_1,X_2)\in \RR_{>0}^2$.
By adapting an 
argument of Matthiesen \cite[Lemma 3.1]{M} we will show that $M_d(X)$ and $N_d(\XX)$ are both small for large square-free values of $d$. This is the purpose of the following result, which is absolutely pivotal in our work.

\begin{lemma}\label{lem:M}
Let $\ve>0$ and let  $d\in \NN$ be square-free. Then 
$$
M_d(X)\ll_\ve \frac{X}{d^{2-\ve}}
\quad \text{ and } \quad
N_d(\XX)\ll_\ve \frac{X_1X_2}{d^{2-\ve}},
$$
where the implied constants depend at most on  $K$ and the choice of $\ve$.
\end{lemma}

\begin{proof}
Put $R(m):= \#\{\x\in \ZZ^n\cap \fF: \nf_{K/\QQ}(\x)=m\}$, for any integer $m$. It follows from \cite[Lemma 8.1]{BM} that 
$
R(m)\ll  r_K(|m|),
$
where $r_K$  is  the multiplicative arithmetic function that appears as  coefficients of the Dedekind zeta function $\zeta_K(s)$.  

It follows that 
\begin{align*}
M_d(X)=\sum_{\substack{|m|\leq X\\ d^2\mid m}}R(m) 
\ll \sum_{\substack{m\leq X\\ d^2\mid m}} r_K(m)
= \sum_{\substack{m\leq X/d^2}} r_K(d^2m).
\end{align*}
Let us factorise $m=hm'$, where $h\mid d^\infty$ and $m'$ is coprime $d$. Then we have
$r_K(d^2m)=r_K(d^2h)r_K(m')$
by multiplicativity.
Hence
\begin{align*}
M_d(X)\ll 
\sum_{h\mid d^\infty } 
r_K(d^2h)
\sum_{\substack{m'\leq X/(d^2h)}} r_K(m').
\end{align*}
The inner sum is $O(X/(d^2h))$ by \cite[Eq.~(2.8)]{BM}. Furthermore,  for any $\ve>0$, we have 
$r_K(d^2h)\leq \tau(d^2h)\ll_\ve (dh)^{\ve/2}$ 
by \cite[Eq.~(2.10)]{BM}. Hence it follows that 
$$
M_d(X)\ll_\ve X \sum_{h\mid d^\infty} \frac{(dh)^{\ve/2}}{d^2h} \ll_\ve  \frac{X}{d^{2-\ve}},
$$
as required, 
since 
$$
 \sum_{h\mid d^\infty} \frac{h^{\ve/2}}{h} \leq
 \sum_{h\mid d^\infty} \frac{1}{\sqrt{h}} =
 \prod_{p\mid d} \left(1-\frac{1}{\sqrt{p}}\right)^{-1}\ll_\ve d^{\ve/2}.
$$

In a similar fashion we find that 
\begin{align*}
N_d(\XX)&=\sum_{\substack{|m_1|\leq X_1\\ d\mid m_1}}R(m_1) 
\sum_{\substack{|m_2|\leq X_2\\ d\mid m_2}}R(m_2) \\
&\ll 
\sum_{h_1,h_2\mid d^\infty } 
r_K(dh_1)r_K(dh_2)\prod_{i=1,2}
\sum_{\substack{m_i'\leq X_i/(dh_i)}} r_K(m_i')\\
&\ll_\ve  \frac{X_1X_2}{d^{2-\ve}}
\end{align*}
This completes the proof of the lemma. 
\end{proof}

\subsection{Reduction to small moduli}

Armed with Lemma \ref{lem:M},  we  now return to our analysis of 
$\cN(G,U,V,W)$, as defined in \eqref{eq:msri},
with the aim of handling the constraint 
$\mu_S^2(N_{L/\QQ}(xy)\str(x,y))=1$.
We will henceforth take 
\begin{equation}
\label{eq:UVW}
U=H^{n_2}V_0^{n_2}, \quad V=V_0^{n_2}, \quad W=H^{n_1/2}V_0^{n_1}.
\end{equation}
for  $V\ge H\ge 1$. 
 One sees immediately that $(UV)^{n_1}=W^{2n_2}$ with these choices. 
 Note that we will ultimately take 
$G=\log V$. 

We recall the convention regarding
the parameters $H$ and $V$ from \cite{BHB}.    $H$ 
will be taken to be a small fixed power of $V$.
There will be certain points in our argument where 
additional  factors of  
$V^{\eta}$ will appear with arbitrary small 
$\eta>0$. This will not matter since we will ultimately make a key saving 
which is a power of $H$, so that the error term makes 
a satisfactory overall contribution.  Let us henceforth write 
$\cN(G,U,V,W)=\cN(G,H,V)$
to better reflect our choice of $U$ and $W$. Thus
$$
\cN(G,H,V)
=
\sum_{\substack{x,y \in \Lring}}
\mu_S^2\left(N_{L/\QQ}(xy) \str(x,y)\right)
\alpha(x)\beta(y)\la\big(\str(x,y)\big).
$$
In view of \eqref{eq:AL}, we see that the function $\alpha$ is supported on 
$x_1+\tau x_2\in \Lring$ such that $x_1,x_2\ll U^{n_1/2}$. Similarly, \eqref{eq:BE} and \eqref{eq:LA} show that $\beta$ (resp.~$\lambda$) is supported on  
$y_1+\tau y_2\in \Lring$ such that $y_1,y_2\ll V^{n_1/2}$
(resp.~ $l\in \ZZ$ such that $l\ll W^{n_2}$).

Lemma \ref{lem:upper-abl} yields
$\cN(G,H,V)\ll_\eta (UV)^{n_1+\eta} W^{\eta}\ll _\eta H^{n_1n_2}V_0^{2n_1n_2+O(\eta)},
$
for any $\eta>0$.
The remainder of this  paper is dedicated to proving a commensurate lower bound, as follows.

\begin{theorem}\label{thm:2}
Let $G=\log V$ and let $H=V^{\frac{1}{\na\nb^2(\na+16)}}$. Then 
we have 
$$
\cN(G,H,V)
\gg (\log V)^{1-2\na-\nb} H^{n_1n_2}V_0^{2n_1n_2}.
$$
\end{theorem}

We begin by simplifying the square-freeness constraint.

\begin{lemma}\label{lem:coprime}
For any   $x,y\in \fo_L$ we have 
$$
\mu_S^2\left(N_{L/\QQ}(xy)\str(x,y)\right)=
\mu_S^2\left(N_{L/\QQ}(xy)\right)\mu_S^2\left(\str(x,y)\right).
$$
\end{lemma}
\begin{proof}
We write everything out in terms of the basis $\{1,\tau\}$ for $\fo_L$, 
so that $x=x_1+\tau x_2$ and $y=y_1+\tau y_2$.
Thus gives
\begin{align*}
N_{L/\QQ}(x)
&=x_1^2+x_2^2 N_{L/\QQ}\tau+ x_1x_2 \tr_{L/\QQ}\tau,\\
N_{L/\QQ}(y)
&=y_1^2+y_2^2 N_{L/\QQ}\tau+ y_1y_2 \tr_{L/\QQ}\tau,\\
\str(x,y)&= x_2y_1-x_1y_2,
\end{align*}
since $\tr_{L/\QQ}(1/(\tau-\tau^\sigma))=0$
and  $\tr_{L/\QQ}(\tau/(\tau-\tau^\sigma))=1$.
Assuming that $\mu_S^2(N_{L/\QQ}(xy))=1$, 
it suffices to show that  there is no common prime divisor $p$ of $\str(x,y)$ and $N_{L/\QQ}(x)$ such that $p\nmid \Sn$.  To see this, we suppose otherwise for a contradiction. We may assume that one of $y_1$ or $y_2$  is coprime to $p$, else we would have $p^2\mid N_{L/\QQ}(y)$. 
 Suppose that $p\nmid y_1$ and let $\bar{y_1}$ denote the multiplicative inverse of $y_1$ modulo $p$. Then 
$p\mid \str(x,y)$ implies that  $x_2\equiv x_1\bar y_1y_2 \bmod{p}$. Next, we deduce from $p\mid N_{L/\QQ}(x)$ that 
$$
x_1^2\left(1+(\bar y_1 y_2)^2 N_{L/\QQ}\tau +\bar y_1 y_2 \tr_{L/\QQ}\tau\right) \equiv 0 \bmod{p}.
$$ 
But this is equivalent to $p\mid N_{L/\QQ}(y)$, since we cannot have $p\mid x_1$, which in turn implies that $p^2\mid N_{L/\QQ}(xy)$. This is a contradiction and so  completes the proof of the lemma.
\end{proof}

Given non-zero integers $a,b$, let us write 
$(a,b)_S$
for the greatest common divisor of $a$ and $b$ which is coprime to  $\Sn$, in the notation of \eqref{eq:Sn}.
 Lemma~\ref{lem:coprime} implies
that 
\begin{align*}
\cN(G,H,V)=
\hspace{-0.3cm}
\sum_{\substack{x,y \in \Lring\\ (N_{L/\QQ}(x),N_{L/\QQ}(y))_S=1}}
\hspace{-0.3cm}
&\mu_S^2\left(N_{L/\QQ}(x)\right)
\mu_S^2\left(N_{L/\QQ}(y)\right)\mu_S^2\left(\str(x,y)\right)\\
&\quad \times
\alpha(x)\beta(y)\la\big(\str(x,y)\big).
\end{align*}
We use M\"obius inversion to take care of the coprimality condition and \eqref{eq:identity} to open up the factors involving $\mu_S^2$. This leads to the conclusion that 
\begin{align*}
\cN(G,H,V)=
\sum_{\substack{d,e,f,k=1\\ (defk,\Sn)=1}}^\infty\mu(d)\mu(e)\mu(f)\mu(k)
\sum_{\substack{x,y \in \Lring\\ [d^2,k]\mid N_{L/\QQ}(x)\\ [e^2,k]\mid N_{L/\QQ}(y)\\
f^2\mid \str(x,y)
}}
\alpha(x)\beta(y)\la\big(\str(x,y)\big).
\end{align*}
We separate out the contribution 
\begin{equation}\begin{split}\label{eqn4.10}
\cN_\xi(G,H,V)=~&
\sum_{\substack{d,e,f,k\leq V^\xi \\ (defk,\Sn)=1}}\mu(d)\mu(e)\mu(f)\mu(k)\\
&\quad \times\sum_{\substack{x,y \in \Lring\\ [d^2,k]\mid N_{L/\QQ}(x)\\ [e^2,k]\mid N_{L/\QQ}(y)\\
f^2\mid \str(x,y)
}}
\alpha(x)\beta(y)\la\big(\str(x,y)\big),
\end{split}
\end{equation}
from small values of $d,e,f,k$.
The reader should think of $\xi$ as being a fixed but  small positive real number. 
Our next task is to show that 
\begin{equation}\label{eqn4.10b}
\cN(G,H,V)=\cN_\xi(G,H,V)+O_\eta\left(H^{n_1n_2}V_0^{2n_1n_2-\xi+O(\eta)}\right),
\end{equation}
for any $\eta>0$.  This means, on taking $\eta$ sufficiently small in terms of $\xi$,  that 
in order to prove Theorem \ref{thm:2} 
we may henceforth focus on 
a lower bound for 
$\cN_\xi(G,H,V)$ for any fixed $\xi>0$.

To establish the claim let us first  consider  the overall  contribution from $d>V^{\xi}$. Since $\alpha,\beta, \lambda$ are supported away from zero, this means that there are $O_\eta(V^\eta)$ choices for $e,f,k$ for a fixed choice of $x,y$ appearing in the sum, by the standard estimate for the divisor function. Applying  Lemma~\ref{lem:upper-abl} to estimate $\beta$ and $\lambda$, we open up $\alpha$ to find that this contribution is 
 \begin{align*}
&\ll_{\eta} V^{n_1+O(\eta)} \sum_{d>V^{\xi}} \mu^2(d) \sum_{\substack{x\in \Lring\\ d^2\mid N_{L/\QQ}(x) }}\alpha(x)
\ll_{\eta} V^{n_1+O(\eta)} \sum_{d>V^{\xi}} \mu^2(d) M_{d}(cU^{n_1}),
\end{align*} 
in the notation of  \eqref{eq:def-MN}, for an appropriate constant $c>0$ depending on $K_1$.
But now we invoke Lemma \ref{lem:M} to bound this by 
$$
\ll_{\eta} V^{n_1+O(\eta)} \sum_{d>V^{\xi}} \frac{U^{n_1}}{d^{2-\eta}} \ll_\eta (UV)^{n_1}V^{-\xi+O(\eta)},
$$
as claimed. The same argument deals with the contribution from 
$e$  exceeding $V^{\xi}$.  
Similarly, on applying  Lemma~\ref{lem:upper-abl} to estimate $\lambda$, we open up $\alpha$ and $\beta$ to find that the contribution from $k>V^{\xi}$ is 
 \begin{align*}
&\ll_{\eta} V^{O(\eta)} \sum_{k>V^{\xi}} \mu^2(k) \sum_{\substack{x,y\in \Lring\\ k\mid (N_{L/\QQ}(x),N_{L/\QQ}(y)) }}\alpha(x)\beta(y)\\
&\ll_{\eta} V^{O(\eta)} \sum_{k>V^{\xi}} \mu^2(k) N_{k}(cU^{n_1},cV^{n_1}),
\end{align*} 
for an appropriate constant $c>0$, in the notation of  \eqref{eq:def-MN}.  Lemma \ref{lem:M} therefore 
shows that this makes an overall contribution 
$O_\eta( (UV)^{n_1}V^{-\xi+O(\eta)})$, as required.
Finally, using the same strategy, the contribution from $f>V^{\xi}$ is seen to be
 \begin{align*}
%&\ll_{\eta} V^{O(\eta)} \sum_{f>V^{\xi}} \mu^2(f) \sum_{\substack{x,y\in \Lring\\ \alpha(x)\beta(y)\neq 0\\ f^2\mid \str(x,y)}} \lambda(\str(x,y))\\
&\ll_{\eta} V^{O(\eta)} \sum_{f>V^{\xi}} \mu^2(f) 
\sum_{\substack{\w \in \cW\cap \ZZ^{n_2} \\ f^2\mid c \nf_{K_2/\QQ}(\w)}} 
N(\w),
\end{align*} 
where $N(\w)$ denotes the number of $x,y\in \Lring$ for which $\alpha(x)\beta(y)\neq0$ and 
$\str(x,y)=c\nf_{K_2/\QQ}(\w)$. Writing $x=x_1+\tau x_2$ and $y=y_1+\tau y_2$,  
so  that $\str(x,y)=x_2y_1-x_1y_2$, 
the condition 
$\alpha(x)\beta(y)\neq0$ ensures that $x_1,x_2\ll U^{n_1/2}$ and $y_1,y_2\ll V^{n_1/2}$. Hence it follows from the standard  estimate for the divisor function that 
$N(\w)=O_\eta(U^{n_1/2}V^{n_1/2+\eta})$. Applying  Lemma 
\ref{lem:M} to estimate the number of $\w$ we obtain the contribution
 \begin{align*}
&\ll_{\eta} (UV)^{n_1/2}V^{O(\eta)} \sum_{f>V^{\xi}} \mu^2(f)  \frac{W^{n_2}}{f^{2-\eta}}
\ll_{\eta} (UV)^{n_1}V^{-\xi +O(\eta)},
\end{align*} 
which is also satisfactory. This completes the proof of \eqref{eqn4.10b}.

\subsection{Preliminary analysis of $\cN_\xi(G,H,V)$}

In this section we prepare the evaluation of the main term  \eqref{eqn4.10}. We split the contribution according to the residue classes of  $\u$, $\v$ and $\w$. In the following we write $$m=[d^2,e^2,f^2,k].$$ In particular, we have $(m,\Sn)=1$ and $m\leq V^{7\xi}.$
Let $\cS (d,e,f,k)$ be the set of residue classes $\u,\v \in  (\ZZ/m\ZZ)^{n_1}$ and $\w \in  (\ZZ/m\ZZ)^{n_2}$ such that 
$$ 
[d^2,k] | N_{L/\QQ} (\del \nf_{K_1/L}(\u)),\quad [e^2,k]| N_{L/\QQ}(\nf_{K_1/L}(\v)D_L),\quad f^2|c\nf_{K_2/\QQ}(\w).
$$
Let $(\u^{(m)}, \v^{(m)},\w^{(m)})\in \cS(d,e,f,k)$.
We denote by $\u^{(M,m)}\in (\ZZ/mM\ZZ)^{n_1}$ the vector which reduces to $\u^{(M)}$ modulo $M$ and to $\u^{(m)}$ modulo $m$, and similarly for $\v^{(M,m)}, \w^{(M,m)}$. 
Furthermore,  as in \S \ref{sec2}, we define the counting functions 
\begin{align*}
\alp_{m,\u^{(m)}}(x)&:=
\#
\{
\u\in \cU\cap\ZZ^{n_1}: 
\u\equiv \u^{(M,m)} \bmod{Mm},~ \delta \nf_{K_1/L}(\u)=x
\},\\
\beta_{m,\v^{(m)}}(y)&:=\#\left\{
\v\in \cV\cap\ZZ^{n_1}:  
\begin{array}{l}
\v\equiv \v^{(M,m)} \bmod{Mm},\\ (\nf_{K_1/L}(\v)D_L)^{\sigma}=y,\ (\ref{gcdv}) \mbox{ holds}
\end{array}
\right\},
\end{align*}
for $x,y \in \Lring$, and
$$
  \la_{m,\w^{(m)}}(l):=
\#
\left\{
\w\in \cW\cap\ZZ^{n_2}:  
\w\equiv \w^{(M,m)} \bmod{Mm},~ c\nf_{K_2/\QQ}(\w)=l
\right\},
$$
for  $l\in \ZZ$. Furthermore, we define
\begin{equation}\label{eq:gull}
\cN_{m,\u^{(m)},\v^{(m)},\w^{(m)}}=
\sum_{x ,y\in \Lring }\alpha_{m,\u^{(m)}}(x)\beta_{m,\v^{(m)}}(y)\la_{m,\w^{(m)}}\big(\str(x,y)\big).
\end{equation}
The dependence of this function on $G,H,V$ is to be understood implicitly. Then we have
\begin{equation}\label{eqn5.7}
\begin{split}
\cN_\xi(G,H,V)=~&
\sum_{\substack{d,e,f,k\leq V^\xi \\ (defk,\Sn)=1}}\mu(d)\mu(e)\mu(f)\mu(k)\\ 
&\quad \times \sum_{(\u^{(m)},\v^{(m)},\w^{(m)})\in \cS(d,e,f,k)}\cN_{m,\u^{(m)},\v^{(m)},\w^{(m)}}.
\end{split}
\end{equation}
Our next goal is to evaluate $\cN_{m,\u^{(m)},\v^{(m)},\w^{(m)}}$ for fixed $m$ and fixed vectors $\u^{(m)}$, $\v^{(m)}$, and $\w^{(m)}$.

\section{Approximating functions}

In this section we use results from \S 4 and \S 5 from \cite{BHB} to find an approximation $\alpmuhat$ to $\alpmu$. The main difference is that in \cite{BHB} the modulus $M$ was fixed and all implied constants were allowed to depend on it. In our situation $\alpmu$ also  depends on $m$, which varies in our argument. Hence we need to   indicate explicitly the dependence on $m$
 in all our estimates. As the proofs are almost the same, we omit most of the details. 

We start by constructing a function $\omemu(x)$ such that the conditions  \cite[Eqs.~(4.8)--(4.11)]{BHB} are satisfied. For this let 
\begin{equation}\label{eq:snow}
\omemu(x):= (Mm)^{-\na}\ome_1 (x),
\end{equation}
where $\ome_1(x)$ is the function constructed and described in  \cite[Lemma 9]{BHB}, with associated parameters $M=1$ and $n=\na$. 

%\begin{lemma}[\cite{BHB}]\label{lem9}
%There exists a continuously differentiable function $\ome_1: \R^2\rightarrow [0,\infty)$, depending on $\mathcal{U}$, with $\ome_1(x+h)-\ome_1(x)\ll U^{-\na/2}|h|$, for all $x,h\in \R^2$, and such that
%$$\int_R \ome_1(x_1,x_2)\d x_1 \d x_2 = \meas \{\u\in \mathcal{U}: \del \nf_{K_1/L}(\u)\in R\},$$
%for every square $R$ with sides parallel to the $x_1$ and $x_2$ axes and side-length $\rho\ll U^{\na/2}$. Furthermore, $\ome_1$ is supported on a disc of radius $O(U^{\na/2})$ and satisfies $\ome_1(x)\ll 1$ throughout this disc. Moreover, there is a disc of radius $\gg G^{-1}U^{\na /2}$ around the point $U^{\na/2}\del \nf_{K_1/L}(\u^{(\R)})$ on which  $\ome_1(x)\gg G^{2-\na}$.
%\end{lemma}

 Let $\alpmuhat$ be defined as  \cite[Eq.~(4.13)]{BHB} with $\ome (x)$ replaced by $\ome_m(x)$, in the notation of \eqref{eq:snow}, and with the  density function
$$
\rho^{(m)} (y,q):= \frac{(Mm)^{\na}}{[Mm,q]^{\na}}\card\left\{\s \bmod [Mm,q]:\ 
\begin{array}{l}
\s \equiv \u^{(M,m)}\bmod Mm\\
 \del \nf_{K_1/L}(\s)\equiv y \bmod q
 \end{array}{}
 \right\}.
$$
Thus we have
$$%\begin{equation}\label{eq:hat-alpha}
\alpmuhat = \omemu(x)\sum_{q\leq Q} ~\sumstar_{t\mmod q} e_q^{(L)}(-tx)\sum_{z\mmod q}\rho^{(m)}(z,q)e_q^{(L)}(tz),
$$
where the summation $\sum_{t\mmod q}^*$ is understood as a summation of $t=t_1+t_2\tau$ with $t\in \Z/q\Z$ and $(q,t_1,t_2)=1$, and we recall that 
$e_q^{(L)}(x):=e_q(b)$ 
for $x=a+b\tau\in L$ with $a,b\in \Q$ and $q\in \N$.

By rescaling, we see that condition (4.11) in \cite{BHB} is satisfied for the parameters 
$W\ll (Mm)^{-\na}U^{\na -1}$ and $Q\leq U^{1/2}$. Both of the conditions (4.9) and (4.10) in \cite{BHB} are clear from the definition of the densities $\rho^{(m)}(y,q)$.
%; i.e. we have
% $\rho^{(m)}(0,1)=1$ and 
%$$%\begin{equation}\label{rhosums}
%\sum_{\substack{y\bmod rs\\ y\equiv z \bmod r}}\rho^{(m)}(y,rs)= \rho^{(m)}(z,r).
%$$
Moreover, a short calculation reveals that condition (4.8) in \cite{BHB} holds with 
$E\ll_\eta m^{\na +1} Q^{\na +1} U^{\na -1 +\eta}.$
Let $\alpo := \alpmu-\alpmuhat$. An application of \cite[Lemma~7]{BHB} gives the following lemma.

\begin{lemma}\label{lem10}
Let $R$ be a square in the $(x_1,x_2)$-plane with sides parallel to the coordinate axes and with side length $\rho\geq 1$ satisfying $\rho\ll U^{\na/2}$. Then 
\begin{equation*}
\left| \sum_{\substack{x\in R\\ x\equiv y \bmod q}} \alpo \right| \ll_\eta m^{\na +1}Q^{\na +1}U^{\na-1+\eta},
\end{equation*}
for any $\eta >0$, $\na \geq 3$ and all $y$ modulo $q$ for $q\leq Q\leq U^{1/2}$.
\end{lemma}

Moreover, \cite[Lemma 8]{BHB} yields the following $L^2$-bound for $\alpmuhat$.

\begin{lemma}\label{lem11}
For any $\eta >0$ one has
\begin{equation*}
\sum_{x\in \Lring}|\alpmuhat|^2\ll_\eta U^{\na+\eta}G^{2\na},
\end{equation*}
for $Q^2\leq U$, $G\leq U^{1/(\na +1)}$ and $Q^{\na +5} (Mm)^{\na +1}\leq U^{1-\eta}$. 
\end{lemma}
%detailed calculation:\sum_{x\in \Lring}|\alpmuhat|^2\ll_\eta (\frac{N^2}{U^{\na}G^{-\na}})^2 (N^2+U^{\na +\eta}$, now choose $N^2=U^{\na}$. 
Finally, we observe that
\begin{equation*}
\begin{split}
\sum_{x\in \Lring} |\alpo|^2&\ll \sum_{x\in \Lring}|\alpmu|^2 + \sum_{x\in \Lring} |\alpmuhat|^2 \ll_\eta U^{\na +\eta} G^{2\na}.
\end{split}
\end{equation*}
Next we apply \cite[Lemma 13]{BHB} to the function $\alpo$. Here one can verify that the parameters $W_0=m^{\na +1}Q^{\na +1}U^{\na -1+\eta}$ (which is a consequence of Lemma \ref{lem10}) and $A_0\ll_\eta Q^3U^\eta$ form an admissible choice. We obtain the following result.

\begin{lemma}\label{lem14}
Let $\eta>0$, $Q\leq Q_0\leq U^{1/(\na +16)}$ and $G\leq U^{1/(\na +1)}$. Then 
$$
\sum_{q\leq Q_0}q^2\sum_{y\bmod q}\max_R\left|\sum_{\substack{x\in R\\ x\equiv y \bmod q}}\alpo\right|^2
\ll_\eta Q_0Q^{-1}(Mm)^{2(\na+1)}U^{2\na + 6\eta}G^{2\na}.
$$
\end{lemma}
%more precise computation: RHS\ll Q_0^{13}Q^{2\na +18}(Mm)^{\na +1}U^{2\na -2+6\eta} + Q_0Q^{-1}U^{2\na +\eta}G^{2\na}
We now consider the error term
$$
\cE_{m,\u^{(m)},\v^{(m)},\w^{(m)}}=
\sum_{x ,y\in \Lring }\alpo \beta_{m,\v^{(m)}}(y)\la_{m,\w^{(m)}}\big(\str(x,y)\big).
$$
As in \cite[Lemma 15]{BHB} the above estimates allow us to provide the following bound on $\cE_{m,\u^{(m)},\v^{(m)},\w^{(m)}}$.

\begin{lemma}\label{lem16}
Let $G=\log V$, let $m\leq V^{7\xi}$ and assume that
\begin{equation}\label{eqnQ1}
(\log V)^{16}\ll Q\ll \min\{ H^{4\na\nb/7}, H^{-4\na\nb}U^{8/(\na +16)}\}.\end{equation}
Then 
$$\cE_{m,\u^{(m)},\v^{(m)},\w^{(m)}}\ll_\eta Q^{-1/16} m^{(\na +1)/2}V_0^{2\na\nb +O(\eta)}H^{\na\nb} .$$
\end{lemma}
%removed the proof as should be considered standard.

\section{Evaluation of the main term}\label{s:6}

% most of this is very close to \cite{BHB}. So I shortened this a lot (wanting to draw the attention of the reader to the singular series section which is more different) as the calculations are minor adaptions of \cite{BHB}. Again the original version is preserved in nqa v4 or v5

In this section we evaluate the main term contributing to the counting function 
\eqref{eq:gull}, which is given by
$$
\calM_{m, \u^{(m)},\v^{(m)},\w^{(m)}}:= \sum_{x,y\in \Lring} \alpmuhat \betmu \lam_{m,\w^{(m)}}(\str(x,y)).
$$
At this point we can still closely follow the arguments in \cite[\S 8]{BHB}. For this we first introduce some notation. We define
$$
F(\v,\w,\s):= 2^{-\kap}\left(\Tr_{L/\Q}(\del \nf_{\Ka/L}(\s)\nf_{\Ka/L}(\v))-c\nf_{\Kb/\Q}(\w)\right),
$$
where $\kap=1$ 
if $2\mid \Tr_{L/\Q}(\tau)$ and $\kap=0$ if $2\nmid \Tr_{L/\Q}(\tau)$.
Moreover, we introduce the notation
\begin{equation*}
a_1=\Tr_{L/\Q}(\nf_{\Ka/L}(\v)),\quad a_2=\Tr_{L/\Q}(\tau \nf_{\Ka/L}(\v))\quad \mbox{ and } \quad b= c\nf_{\Kb/\Q}(\w).
\end{equation*}
In approximating the sums in the main term with integrals one is lead to consider
$$I(\v,\w):= \int_{-\infty}^\infty \omemu (a_2x,-a_1x+b/a_2)\d x,$$
if $a_2\neq 0$. (If  $a_2=0$ and $a_1\neq 0$ then it is to be understood that one replaces the integrand by $\omemu(-a_2x+b/a_1,a_1x)$ in this definition.)\par
At the finite places one is lead to considering the system of congruences
\begin{equation}\label{eqnMm}
(\p,\q,\s)\equiv (\v^{(M,m)},\w^{(M,m)},\u^{(M,m)}) \mmod Mm,
\end{equation}
and
\begin{equation}\label{eqnF}
F(\p,\q,\s)\equiv 0\mmod u,\quad \mbox{ and }\quad l|\nf_{\Ka/L}(\p).
\end{equation}
We write $\Del=[Mm,u,l]$. For given positive integers $l,u$, we define the counting function $N_{M,m}(l,u)$ to be the number of tuples $(\p,\q,\s)$ modulo $\Del$ such that (\ref{eqnMm}) and (\ref{eqnF}) hold.

We now define the truncated singular series 
$$\tilde{\grS}^{(m)}(Q):=\sum_{q\leq Q}\sum_{l=1}^\infty\mu(l)\sum_{u|q}\frac{u\mu(q/u)}{\Del^{2\na+\nb}}N_{M,m}(l,u)$$
and the singular integral
\begin{equation}\label{eq:thrush}
\sig_\infty:=m^{\na}\sum_{\w\in \calW\cap\Z^{\nb}}\sum_{\v\in \calV\cap\Z^{\na}}I(\v,\w).
\end{equation}
With this notation, we obtain the following approximation for our main term. As the proof closely follows the arguments \cite[\S 8]{BHB}, we omit its details. 

\begin{lemma}\label{lem19}
Assume that 
$
H^{\na\nb/2}\leq V$ and $H^{\na\nb/2-\na/2}<V_0^{\na}.$
Then there is a constant $\vartet_1$ such that 
\begin{equation*}
\calM_{m,\u^{(m)},\v^{(m)},\w^{(m)}}= 2^{\kap}M^{\na}\sig_\infty \tilde{\grS}^{(m)}(Q)+O_\eta\big(m^{\vartet_1}Q^6H^{\na\nb/2}V_0^{2\na\nb+\eta}\big).
\end{equation*}
In this estimate one can take $\vartet_1=3\na+2\nb+1$. 
\end{lemma}

It is now time to reintroduce the summation  over $d,e,f,k$. 
Recalling the expression for $\calN_\xi (G,H,V)$ in  \eqref{eqn5.7}, and observing that 
$$\calN_{m,\u^{(m)},\v^{(m)},\w^{(m)}}=\calM_{m,\u^{(m)},\v^{(m)},\w^{(m)}} +
\calE_{m,\u^{(m)},\v^{(m)},\w^{(m)}},$$
we assume that $\xi$ is sufficiently small and that (\ref{eqnQ1}) holds.
Then it follows from  Lemma \ref{lem16} that
\begin{align*}
\calN_\xi(G,H,V)=~& 
\hspace{-0.1cm}\sum_{\substack{d,e,f,k\leq V^\xi \\ (defk,\Sn)=1}}
\hspace{-0.1cm}
\mu(d)\mu(e)\mu(f)\mu(k)
\hspace{-0.7cm}\sum_{(\u^{(m)},\v^{(m)},\w^{(m)})\in \cS(d,e,f,k)}
\hspace{-0.7cm}\calM_{m,\u^{(m)},\v^{(m)},\w^{(m)}}\\
&\quad +O_\eta\left(Q^{-1/16}V^{7\xi(\na+1)/2}V^{4\xi+7\xi(2\na+\nb)}V_0^{2\na\nb+O(\eta)}H^{\na\nb}\right),
\end{align*}
where we recall the notation \eqref{eq:Sn} for $\Sn$.
We now use Lemma \ref{lem19} to replace  $\calM_{m,\u^{(m)},\v^{(m)},\w^{(m)}}$ with the expected approximation, giving 
\begin{equation}\label{eqnNdel}
\begin{split}
\calN_\xi (G,H,V)&-
\calN_\xi^{(1)}(G,H,V)\\
\ll_\eta ~& Q^{-1/16}V^{7\xi(\na+1)/2}V^{4\xi+7\xi(2\na+\nb)}V_0^{2\na\nb+O(\eta)}H^{\na\nb}\\
&+V^{4\xi+7\xi(2\na+\nb)}V^{7\xi\vartet_1}Q^6H^{\na\nb/2}V_0^{2\na\nb+\eta},
\end{split}
\end{equation}
where 
\begin{align*}
\calN_\xi^{(1)}(G,H,V):=~& \sum_{\substack{d,e,f,k\leq V^\xi \\ (defk,\Sn)=1}}\mu(d)\mu(e)\mu(f)\mu(k)\hspace{-0.7cm}\\
&\quad \times\sum_{(\u^{(m)},\v^{(m)},\w^{(m)})\in \cS(d,e,f,k)}\hspace{-0.7cm}2^{\kap}M^{\na}\sig_\infty \tilde{\grS}^{(m)}(Q).
\end{align*}
We will be interested in triples of vectors $(\p,\q,\s)$ such that 
\begin{align}\label{eqnM}
(\p,\q,\s)\equiv (\v^{(M)},\w^{(M)},\u^{(M)}) \mmod M, \quad 
f^2\mid c\nf_{\Kb/\Q}(\q),\\
\label{defk}
[d^2,k]\mid N_{L/\Q}(\del \nf_{\Ka/L}(\s)),\quad
[e^2,k]\mid  N_{L/\Q}(\nf_{\Ka/L}(\p)D_L).
\end{align}
Letting $\Del:=[d^2,e^2,f^2,k,M,l,u]$, we define the counting function
\begin{equation}\label{eq:chick}
R(d,e,f,k,l,u):=\#\left\{(\p,\q,\s)\bmod \Del: 
\text{\eqref{eqnF}, \eqref{eqnM},  \eqref{defk} hold}\right\}.
\end{equation}
Then we have
\begin{equation}\label{defN1}
\begin{split}
\calN_\xi^{(1)}(G,H,V)=~&2^{\kap}M^{\na}\sig_\infty  \sum_{\substack{d,e,f,k\leq V^\xi \\ (defk,\Sn)=1}}\mu(d)\mu(e)\mu(f)\mu(k) \\ &\quad \times \sum_{q\leq Q}\sum_{l=1}^\infty\mu(l)\sum_{u|q}\frac{u\mu(q/u)}{\Del^{2\na+\nb}}R(d,e,f,k,l,u).
\end{split}
\end{equation}
We would next like to demonstrate the positivity of the sum over $d,e,f,k,l$ and $q$. This is hampered by the presence of the M\"obius functions and an important first step will be to show that the truncated sum in \eqref{defN1} can be completed.
This is the object of the next section.

\section{The singular series}\label{s:7}

In this section we continue our analysis of the sum \eqref{defN1} and then 
bring everything together in order to   complete the proof of Theorem 
 \ref{thm:2} (and so the proof of Theorem \ref{thm:1}).
We begin  with a simple upper bound for the counting function  \eqref{eq:chick}, which does not exploit the 
congruence condition \eqref{eqnF} modulo $u$, but is nonetheless useful when $u$ is fixed.

\begin{lemma}\label{lemR}
Let $\Del=[d^2,e^2,f^2,k,M,l,u]$ and let $\eta>0$. Then 
\begin{equation*}
R(d,e,f,k,l,u)\ll_\eta \frac{\Del^{2\na+\nb}}{[d^2,k]^{1-\eta}[l^2,e^2,k]^{1-\eta}f^{2(1-\eta)}}.
\end{equation*}
\end{lemma}

\begin{proof}
Let $\Del':=[d^2,e^2,f^2,k,M,l^2,u]$. Then $R(d,e,f,k,l,u)$ is at most
\begin{equation*}
\left(\frac{\Del}{\Del'}\right)^{2\na+\nb} 
\hspace{-0.2cm}
\#
\left\{(\p,\q,\s)\bmod{\Del'}: 
 \begin{array}{l}
 ~[d^2,k]\mid N_{L/\Q}(\del)\nf_{\Ka/\Q}(\s)  \\
~[l^2,e^2,k] \mid  N_{L/\Q}(D_L)\nf_{\Ka/\Q}(\p) \\
~f^2\mid c\nf_{\Kb/\Q}(\q) 
 \end{array}
\right\}.
\end{equation*}
The cardinality  on the right hand side factors into three independent counting functions for $\p$, $\q$ and $\s$. By  \cite[Lemma 4.2]{BM} we have
\begin{align*}
\card\{\p\mmod \Del': [l^2,e^2,k]\mid &N_{L/\Q}(D_L)\nf_{\Ka/\Q}(\p)\}
\hspace{-0.1cm}
\ll_\eta 
\hspace{-0.1cm}\left(\frac{\Del'}{[l^2,e^2,k]}\right)^{\na} 
\hspace{-0.3cm}
[l^2,e^2,k]^{\na-1+\eta}.
\end{align*}
In a similar way one can estimate the contributions from $\s$ modulo $\Del'$ with $[d^2,k]\mid N_{L/\Q}(\del)\nf_{\Ka/\Q}(\s)$, and $\q$ with $f^2\mid c\nf_{\Kb/\Q}(\q)$. Together these bounds imply 
\begin{equation*}
R(d,e,f,k,l,u)\ll_\eta \frac{\Del^{2\na+\nb}}{[d^2,k]^{1-\eta}[l^2,e^2,k]^{1-\eta}f^{2(1-\eta)}},
\end{equation*}
as desired.
\end{proof}

In order to interpret the singular series we will need to complete the summation over $q$. Estimating the sum over $u\mid q$ trivially by taking absolute values of $\mu(q/u)$ does not suffice, even if the right decay for $R(d,e,f,k,l,u)$ is established, since the resulting majorant would not be convergent. Hence we proceed by consider the function
$$
g(d,e,f,k,l,q):= \sum_{u|q} \frac{u\mu(q/u)}{\Del^{2\na+\nb}}R(d,e,f,k,l,u),
$$
where $R(d,e,f,k,l,u)$ is given by \eqref{eq:chick}.
Note that $g$ and $R$ implicitly depend on $M$. When we need to articulate the  specific value of $M$, we will  write $g_M$ instead of $g$ and $R_M$ instead of $R$.

\subsection{Estimation of $g$}

Assume that we can write $d=d_1d_2$, $e=e_1e_2$, $f=f_1f_2$, $k=k_1k_2$, $l=l_1l_2$, $q=q_1q_2$ and $M=M_1M_2$, with
$$(d_1e_1f_1k_1l_1M_1q_1,d_2e_2f_2k_2l_2M_2q_2)=1.$$
Then the function $g$ has the multiplicativity property
$$g_M(d,e,f,k,l,q)= g_{M_1}(d_1,e_1,f_1,k_1,l_1,q_1)g_{M_2}(d_2,e_2,f_2,k_2,l_2,q_2).$$
Hence it is sufficient to study $g_M(d,e,f,k,l,q)$ for values of $M,d,e,f,k,l,q$ which are all powers of the same prime number $p$. Let $q=p^\alp$ and assume now that $M,d,e,f,k,l$ are powers of $p$. We write $\nu_p=\val_p$ for the standard $p$-adic valuation on $\Q$. In our study of the function $g$, we first consider the case where 
\begin{equation}\label{eqnalp}
\alp+1\geq 2\max\{\nu_p([d^2,k]), \nu_p([e^2,k]), \nu_p(f^2),\nu_p(l),\nu_p(M), 3/2\}.
\end{equation}
For $\alpha$ in this range we have 
\begin{equation}\label{eqng}
\begin{split}
g_M(d,e,f,k,l,p^\alp)=~& \frac{p^\alp}{p^{\alp(2\na+\nb)}}R_M(d,e,f,k,l,p^\alp)\\ &\quad- \frac{p^{\alp-1}}{p^{(\alp-1)(2\na+\nb)}}R_M(d,e,f,k,l,p^{\alp-1}).
\end{split}
\end{equation}
In the following we write $\bfx$ for the vector $(\p,\q,\s)$ and define $F(\bfx):=F(\p,\q,\s)$. Then \eqref{eq:chick} becomes
$$
R_M(d,e,f,k,l,p^\alp)= \card\left\{\bfx\bmod p^\alp: 
\begin{array}{l}
F(\bfx)\equiv 0 \bmod p^\alp, ~l\mid \nf_{\Ka/L}(\p)\\ 
 \text{(\ref{eqnM}), (\ref{defk})  hold}
\end{array}{}
\right\}.
$$
For $\tau<\alp$ we define the counting function
$$
R_M^\tau(d,e,f,k,l,p^\alp)= \card\left\{\bfx\bmod p^\alp: 
\begin{array}{l}
F(\bfx)\equiv 0 \bmod p^\alp, ~l\mid \nf_{\Ka/L}(\p)\\ 
 \text{$p^\tau\|\nabla F(\bfx)$ and (\ref{eqnM}), (\ref{defk})  hold}
\end{array}{}
\right\}.
$$
Then 
$R_M(d,e,f,k,l,p^\alp)=\sum_{\tau=0}^{\alp-1}R_M^{\tau}(d,e,f,k,l,p^\alp)+O(T(p^\alpha)),$
where
\begin{equation}\label{eq:egg}
T(p^t):=\card\{\bfx\mmod p^t: p^t|\nabla F(\bfx)\}.
\end{equation}

We now apply Hensel's lemma (in the form   \cite[Lemma 3.3]{BM}, for example) 
in order to compare $R_M(d,e,f,k,l,p^\alp)$ and $R_M(d,e,f,k,l,p^{\alp-1})$. As soon as $\alp-1\geq 2\tau+1$ and 
\begin{equation}\label{eq:rain}\alp-1\geq \tau +\max\{\nu_p([d^2,k]), \nu_p([e^2,k]), \nu_p(f^2),\nu_p(l),\nu_p(M)\},
\end{equation}
we have
$p^{2\na+\nb-1}R_M^{\tau}(d,e,f,k,l,p^{\alp-1})= R_M^{\tau}(d,e,f,k,l,p^\alp).$
Note that the inequality $\alp-1\geq 2\tau+1$ is equivalent to saying that $\tau\leq \alp/2-1$,
from which it is clear that \eqref{eqnalp} implies \eqref{eq:rain}.
 Together with equation (\ref{eqng}) we obtain
\begin{equation}\label{eqng2}
\begin{split}
g_M(d,e,f,k,l,p^\alp)
&\ll \frac{p^\alp}{p^{\alp(2\na+\nb)}}\left(\frac{p^\alp}{p^{\lfloor \frac{\alp}{2}\rfloor}}\right)^{2\na+\nb} T(p^{\lfloor \frac{\alp}{2}\rfloor}),
\end{split}
\end{equation}
in the notation of \eqref{eq:egg}.
The function $T(p^t)$ has already been studied in \cite[\S 9]{BHB}, where the problem is reduced to the analysis of 
 two independent counting functions $T_1$ and $T_2$, in  such a way that $T(p^t)\leq T_1(p^t)T_2(p^t)$. It follows from  \cite[Eq.~(9.13) and p.~1186]{BHB} that 
$$
T_1(p^t)\ll (4t+1)^{2\na}p^{2\na t-\lceil2t\na/(\na-1)\rceil}
\text{ and }
T_2(p^t)\ll (2t+1)^{\nb} p^{\nb t - \lceil t\nb /(\nb-1)\rceil}.
$$
This gives
$$T(p^t)\ll (4t+1)^{2\na+\nb} p^{(2\na +\nb)t-3t-2},$$
since $\lceil \theta \rceil\geq m+1$ for any real number $\theta$ bigger than an integer $m$.
If we use this bound in (\ref{eqng2}),  we obtain the estimate
$$g_M(d,e,f,k,l,p^\alp)\ll (2\alp+1)^{2\na+\nb}p^{\alp-3\lfloor\frac{\alp}{2}\rfloor -2}.$$
By multiplicativity of $g$ and using trivial bounds for $R_M(d,e,f,k,l,1)$, this estimate continues to hold for any positive integers $M,d,e,f,k,l$, which are not necessarily powers of $p$. We state this result in the following lemma.

\begin{lemma}\label{g1}
Assume that \eqref{eqnalp} holds. Then 
$$g_M(d,e,f,k,l,p^\alp)\ll \alp^{2\na+\nb}p^{\alp-3\lfloor\frac{\alp}{2}\rfloor -2}.$$
\end{lemma}

We will  use this lemma later to handle  cases where $\alp\geq 3$, or $\alp\geq 2$ and $p\nmid def$. 
It remains to bound the function $g$ in situations where $d,e,f,k,l$ are square-free, with either  $\alp=1$, or $\alp=2$ and $p\mid def$. We begin by dealing with the latter situation. 

\begin{lemma}\label{g2}
Let  $\eta >0$. 
If $p\mid def$ then 
$
g_M(d,e,f,k,l,p^2)\ll_\eta p^{-2+\eta}.
$
\end{lemma}

\begin{proof}
Suppose first that either $p^2\mid de$, or $p\mid f$ and $p\mid dekl$, or $p^2\mid f$. We use the definition of $g$ in combination with the bound from Lemma \ref{lemR} to deduce that
\begin{equation*}
\begin{split}
g_M(d,e,f,k,l,p^2)&\ll p^2\frac{R_M(d,e,f,k,l,p^2)}{\Del(p^2)^{2\na+\nb}}+p\frac{R_M(d,e,f,k,l,p)}{\Del(p)^{2\na+\nb}} \\
&\ll_\eta p^2 p^{-4(1-\eta/4)}\ll_\eta p^{-2+\eta}.
\end{split}
\end{equation*}
Here we write $\Del(p^2)$ and $\Del(p)$ instead of $\Del$ to indicate the dependence on $u$.

Next we consider the case $p\mid d$ and $p\nmid ef$. We take a similar strategy and note that  the counting function $R_M(d,e,f,k,l,u)$ is 
\begin{equation*}
\leq  \card\left\{(\p,\q,\s)\bmod \Del: 
\begin{array}{l}
p^2\mid N_{L/\Q}(\del)\nf_{\Ka/\Q}(\s), ~F(\p,\q,\s)\equiv 0\bmod{u}
\end{array}{}
\right\}.
\end{equation*}
In order to bound this counting function, we first count the number of possible choices for $\s$ modulo $\Del$. By  \cite[Lemma 4.2]{BM} one has 
$$
\card\{\s\mmod \Del: p^2\mid N_{L/\Q}(\del)\nf_{\Ka/\Q}(\s)\}\ll_\eta \frac{\Del^{\na}}{p^{2-\eta/2}}.
$$
Now consider $\s$ and $\p$ fixed and ask  how many choices for $\q$ arise in the function $R_M(d,e,f,k,l,u)$. 
Again, we may apply  \cite[Lemma 4.2]{BM}, concluding that
$$R(d,e,f,k,l,p^\alpha)\ll_\eta \frac{\Del^{\na}}{p^{2-\eta/2}} \Del^{\na} \frac{\Del^{\nb}}{p^{\alpha-\eta/2}},$$
for $\alpha\in \{1,2\}$.
Hence the desired bound holds for  $p\mid d$.
The case $p\mid e$ may be treated in the same way.

It remains to bound the function $g$ in the case $p\| f$ and $p\nmid dekl$. 
Since $(f,\Sn)=1$ we must have $p\nmid 2M$.
Hence we need to analyse $g_1(1,1,p,1,1,p^2)$. 
We proceed similarly to  the proof of Lemma \ref{g1}. Let 
$$G(\s,\p):=\Tr_{L/\Q}(\del \nf_{\Ka/L}(\s) \nf_{\Ka/L}(\p))$$ 
and define the counting function
$$\tilde{R}(u):=\card\{\s,\p\mmod u: G(\s,\p)\equiv 0\mmod u\}.$$
Then $g_1(1,1,p,1,1,p^2)$ is
\begin{equation*}
\begin{split}
&= p^{-2(2\na+\nb)}\big(p^2R_1(1,1,p,1,1,p^2)-p R_1(1,1,p,1,1,p)\big) \\
&= p^{-2(2\na+\nb)}\card\{\q\mmod p^2: p^2\mid c\nf_{\Kb/\Q}(\q)\} \big(p^{2}\tilde{R}(p^2)-p^{1+2\na}\tilde{R}(p)\big).
\end{split}
\end{equation*}
Again we use  \cite[Lemma 4.2]{BM}
to control the first counting function,  obtaining 
\begin{equation}\label{h1}
g_1(1,1,p,1,1,p^2)\ll_\eta p^{-2+\eta}\big(p^{2-4\na}\tilde{R}(p^2)-p^{1-2\na}\tilde{R}(p)\big),
\end{equation}
for any $\eta>0$.
We further split the counting function $\tilde{R}(u)$ on prime powers $u=p^\alp$ into two parts. Let 
$$\tilde{R}_1(p^\alp):= \card\{\s,\p\mmod p^\alp: G(\s,\p)\equiv 0\mmod p^\alp,\ p\nmid \nabla G(\s,\p)\}.$$
According to \eqref{h1} we have 
\begin{equation*}
\begin{split}
g_1(1,1,p,1,1,p^2)\ll_\eta~& 
p^{-2+\eta}\big(p^{2-4\na}\tilde{R}_1(p^2)-p^{1-2\na}\tilde{R}_1(p)\big)
\\
& \quad +p^{-2+\eta}p^{2-2\na}\card\{\s,\p \mmod p: p\mid \nabla G(\s,\p)\}.
\end{split}
\end{equation*}
If we define $T_1(p^\tau)$ as in \cite[\S 9]{BHB}, then we see that
\begin{equation*}
\begin{split}
p^{2-2\na}\card\{\s,\p & \mmod p: p\mid \nabla G(\s,\p)\}\ll p^{2-2\na}T_1(p)
\ll p^{2-\lceil 2\na/(\na-1)\rceil}\ll p^{-1},
\end{split}
\end{equation*}
and hence the second term is bounded by $O_\eta (p^{-3+\eta})$.
Finally, it follows from \cite[Lemma~3.3]{BM} that
$p^{2-4\na}\tilde{R}_1(p^2)=p^{1-2\na}\tilde{R}_1(p),$
whence in this case 
$g_1(1,1,p,1,1,p^2)\ll_\eta p^{-3+\eta}$.
This completes the proof of the lemma.
\end{proof}

Finally  we turn to the case $\alp=1$.

\begin{lemma}\label{g3}
Let  $\eta>0$. Then 
$g_M(d,e,f,k,l,p)\ll_\eta p^{-2+\eta}.$
\end{lemma}

\begin{proof}
By using the trivial bound $R(d,e,f,k,l,u)\leq \Del^{2\na+\nb}$ and the multiplicativity of the function $g$, we may assume that all of $d,e,f,k,l$ are powers of $p$. Moreover, since only finitely many primes $p$ divide $M$, and our implicit constants may depend on $M$, we can assume that $M=1$. Recall that
\begin{equation}\label{g4}
g_1(d,e,f,k,l,p)=p\frac{R_M(d,e,f,k,l,p)}{\Del(p)^{2\na+\nb}}- \frac{R_M(d,e,f,k,l,1)}{\Del(1)^{2\na+\nb}}.
\end{equation}
If $p^3\mid [d^2,k][l^2,e^2,k]f^2$, then the bound from Lemma \ref{lemR} is already sufficient to establish the lemma. It remains to consider the following three cases.

\noindent{{\em Case (i): $p\nmid defkl$}}.
In this case \eqref{g4} becomes
$$g_1(1,1,1,1,1,p)= \frac{p}{p^{2\na+\nb}}\card\{\bfx\mmod p: F(\bfx)\equiv 0 \mmod p\}-1.$$
We shall count $\FF_p$ points on the hypersurface $F=0$, which has projective dimension $2n_1+n_2-2$ and 
 singular locus of  projective dimension  $2\na+\nb-7$. Hence  we obtain
$$g_1(1,1,1,1,1,p)= \frac{p}{p^{2\na+\nb}}(p^{2\na+\nb-1}+O(p^{2\na+\nb-3}))-1 \ll p^{-2},$$
which is satisfactory.

\noindent{{\em Case (ii): $p\| defkl$}}.
From Lemma \ref{lemR} it is clear that
$$\frac{R_M(d,e,f,k,l,1)}{\Del(1)^{2\na+\nb}}\ll_\eta p^{-2+\eta}.$$
We start with the case that $d=p$, observing that $p\nmid N_{L/\QQ}(\delta)$ since $(d,S)=1$.
In this case $R_1(p,1,1,1,1,p)$ is equal to
\begin{equation*}
 \card\{(\p,\q,\s)\mmod p^2: p^2\mid \nf_{\Ka/\Q}(\s),~ F(\p,\q,\s)\equiv 0\bmod{p}\}.
\end{equation*}
To estimate this, we first sum over all $\s$ modulo $p^2$ with $p^2\mid \nf_{\Ka/\Q}(\s)$, and then fix some arbitrary $\p$ modulo $p^2$. For such a tuple, one again needs to estimate the number of solutions to $\nf_{\Kb/\Q}(\q)\equiv \mu$ modulo $p$ for some residue $\mu$ modulo $p$. Two applications of \cite[Lemma 4.2]{BM} lead to the bound
$$
R(p,1,1,1,1,p)\ll_\eta p^{2(2\na+\nb)} p^{-2+\eta/2}p^{-1+\eta/2}\ll_\eta p^{2(2\na+\nb)-3+\eta}.
$$
Hence we deduce from  (\ref{g4}) that
$g_1(p,1,1,1,1,p)\ll_\eta p^{-2+\eta},$
as desired. The estimates for the other cases where $p$ divides exactly  one of $e,f,k,l$ are established in the same way.

\noindent{{\em Case (iii):  $p\mid l$ and $p\mid e$ and $p\nmid dfk$}}.
The same arguments as above leads to the estimate
$$R_1(1,p,1,1,p,p)\ll \card\{\x\mmod p^2: p^2|\nf_{\Ka/\Q}(\p),\ F(\x)\equiv 0\mmod p\}\ll_\eta p^{-3+\eta},$$
and
$R_1(1,p,1,1,p,1)\ll_\eta p^{-2+\eta}.$
Once combined with  (\ref{g4}), this is enough to establish the lemma.
\end{proof}

\subsection{Absolute convergence of the singular series}

We have now collected all the bounds that we need to show that one may complete the summations over $d,e,f,k,q$ in the definition of $N_\xi^{(1)}(G,H,V)$ in (\ref{defN1}). Observe that 
\begin{equation*}
\begin{split}
\sum_{\substack{d,e,f,k\leq V^\xi \\ (defk,\Sn)=1}}&|\mu(d)\mu(e)\mu(f)\mu(k) |\sum_{q\leq Q}\sum_{l=1}^\infty|\mu(l)||g_M(d,e,f,k,l,q)| \\ \ll & \prod_p \left(\sum_{\alp_1,\ldots,\alp_5\in\{0,1\}}\sum_{\bet=0}^\infty|g_{p^{\nu_p(M)}} (p^{\alp_1},\dots,p^{\alp_5},p^{\bet})|\right),
\end{split}
\end{equation*}
where the product over $p$ is taken over all prime numbers.
We define
$$ 
\grS:= \sum_{\substack{d,e,f,k\in\N \\ (defk,\Sn)=1}}\sum_{l=1}^\infty\mu(d)\mu(e)\mu(f)\mu(k)\mu(l) \sum_{q=1}^\infty g_M(d,e,f,k,l,q) .
$$
The following lemma shows that $\grS$ is indeed absolutely convergent.

\begin{lemma}
For any prime $p$ one has
$$\sum_{\alp_1,\ldots,\alp_5\in\{0,1\}}\sum_{\bet=0}^\infty|g_{p^{\nu_p(M)}} (p^{\alp_1},\dots,p^{\alp_5},p^{\bet})|= p^{-\nu_p(M)(2\na+\nb)}\big(1+O_\eta (p^{-2+\eta})\big),$$
where the implied constant depends only on $\eta$ and $M$.
\end{lemma}

\begin{proof}
It is easy to see that
$g_{p^{\nu_p(M)}}(1,\dots,1)= p^{-\nu_p(M)(2\na+\nb)}.$
Hence we need to show that the contribution from the remaining  summands is $O_\eta (p^{-2+\eta})$. 
For $\bet=0$ and $\sum_{i=1}^5 \alp_i\geq 1$, Lemma \ref{lemR} shows that
$$g_{p^{\nu_p(M)}} (p^{\alp_1},\dots,p^{\alp_5},1)\ll_\eta p^{-2+\eta}.$$
Furthermore, Lemma \ref{g3} implies that
$$ \sum_{\alp_1,\ldots,\alp_5\in\{0,1\}}|g_{p^{\nu_p(M)}} (p^{\alp_1},\dots,p^{\alp_5},p)|\ll_\eta p^{-2+\eta},$$
which corresponds to the terms with $\bet=1$. For $\bet=2$ and $\alp_1+\alp_2+\alp_3\geq 1$, one may use Lemma \ref{g2} to bound these terms by $O_\eta (p^{-2+\eta})$, and for $\bet=2$ and $\alp_1=\alp_2=\alp_3=0$, the same bound follows from Lemma \ref{g1}. We finally apply Lemma \ref{g1} to estimate the contribution from $\bet\geq 3$ by
\begin{equation*}
\begin{split}
  \sum_{\alp_1,\ldots,\alp_5\in\{0,1\}}\sum_{\bet= 3}^\infty|g_{p^{\nu_p(M)}} (p^{\alp_1},\dots,p^{\alp_5},p^{\bet})|
  &\ll \sum_{\bet=3}^\infty \bet^{2\na+\nb}p^{\bet-3\lfloor\frac{\bet}{2}\rfloor-2} \ll  p^{-2}.
\end{split}
\end{equation*}
This completes the proof of the lemma.
\end{proof}

We now factorise $\grS$ into a product of local densities. For $p\nmid \Sn$ we define 
$$\sig_p:=\sum_{\alp_1,\ldots,\alp_5\in\{0,1\}}(-1)^{\sum_{i=1}^5\alp_i}\sum_{\bet=0}^\infty g_{1}(p^{\alp_1},\dots,p^{\alp_5},p^{\bet}),$$
and for $p\mid \Sn$ we set
$$\sig_p:= \sum_{\alp\in\{0,1\}}(-1)^{\alp} \sum_{\bet=0}^\infty g_{p^{\nu_p(M)}}(1,1,1,1,p^{\alp},p^{\bet}).$$
We have 
$\grS=\prod_p \sig_p,$
since $\grS$ is absolutely convergent. 
In order to show that $\grS$ is positive, it will therefore be sufficient to show that each of the factors $\sig_p$ is positive. 

Note that $\sig_p=\lim_{T\rightarrow \infty} \sig_p(T),$
with
$$ \sig_p(T):= \sum_{\alp_1,\ldots,\alp_5\in\{0,1\}}(-1)^{\sum_{i=1}^5\alp_i}\sum_{\bet=0}^T g_{p^{\nu_p(M)}}(p^{\alp_1},\dots,p^{\alp_5},p^{\bet}),$$
in the case $p\nmid \Sn$, and in a similar way for $p\mid \Sn$. We rewrite $\sig_p(T)$ as
\begin{equation*}
\begin{split}
\sig_p(T)=~ &\sum_{\alp_1,\ldots,\alp_5\in\{0,1\}}(-1)^{\sum_{i=1}^5\alp_i}  \Bigg( \frac{1}{\Del(1)^{2\na+\nb}}R_{p^{\nu_p(M)}}(p^{\alp_1},\ldots, p^{\alp_5},1) \\   &\quad+ \sum_{\bet=1}^T  \frac{p^\bet}{\Del(p^\bet)^{2\na+\nb}}R_{p^{\nu_p(M)}}(p^{\alp_1},\ldots, p^{\alp_5},p^\bet)\\ &\quad -\sum_{\bet=1}^T \frac{p^{\bet-1}}{\Del(p^{\bet-1})^{2\na+\nb}}R_{p^{\nu_p(M)}}(p^{\alp_1},\ldots, p^{\alp_5},p^{\bet-1})\Bigg).
\end{split}
\end{equation*}
We recognise the summation in $\bet$ as a telescoping sum for each fixed vector $(\alp_1,\ldots, \alp_5)$. For $T$ sufficiently large, we therefore obtain
$$\sig_p(T)= \sum_{\alp_1,\ldots,\alp_5\in\{0,1\}}(-1)^{\sum_{i=1}^5\alp_i} \frac{p^T}{p^{T(2\na+\nb)}}R_{p^{\nu_p(M)}}(p^{\alp_1},\ldots, p^{\alp_5},p^T),$$
which we can  interpret as a normalised counting function modulo $p^T$. 
To be precise, for $p\mid \Sn$ and $T\geq \max\{2,\nu_p(M)\}$ define $\calR(p^T)$ to be the number of vectors $\x$ modulo $p^T$ that satisfy 
\begin{equation}\label{cong1}
  (\p,\q,\s)\equiv (\v^{(M)},\w^{(M)},\u^{(M)})\mmod p^{\nu_p(M)},\quad \quad p\nmid \nf_{\Ka/L}(\p),
\end{equation}
and
$$F(\x)\equiv 0 \mmod p^T.$$
For $p\nmid \Sn$ we define $\calR(p^T)$ in the same way, but with the additional restrictions that
\begin{equation}\label{cong2}
\begin{split}
&p^2\nmid N_{L/\Q}(\del \nf_{\Ka/L}(\s)),\quad p^2\nmid N_{L/\Q}(\nf_{\Ka/L}(\p)D_L),\quad p^2\nmid c \nf_{\Kb/\Q}(\q), \\ & p\nmid \left(N_{L/\Q}(\del \nf_{\Ka/L}(\s)) , N_{L/\Q}(\nf_{\Ka/L}(\p)D_L)\right).
\end{split}
\end{equation}
Then for $T\geq \max\{2,\nu_p(M)\}$
one has the expression
$$\sig_p(T)= \frac{p^T}{p^{T(2\na+\nb)}}\calR (p^T).$$

\subsection{Completion of the proof}

 We now have everything in place to show that 
 \begin{equation}\label{eq:pea}
\grS >0.
\end{equation}
Our work in the previous section shows that it suffices to show  that $\sig_p>0$ for every prime $p$. It is clearly 
sufficient to find a smooth $p$-adic solution to $F(\x)=0$, that satisfies the additional congruence conditions (\ref{cong1}) in the case $p\mid \Sn$, and both (\ref{cong1}) and (\ref{cong2}) for $p\nmid \Sn$. 
For $p\nmid M$, it is sufficient to find a non-singular solution to $F(\x)=0$ over $\F_p$, such that $p\nmid \nf_{\Ka/L}(\p)$ and $p$ divides none of the expressions in (\ref{cong2}).

As in  \cite[\S 9]{BHB}, we observe that the equation $F(\x)=0$ has $$
p^{2\na+\nb-1}+O(p^{2\na+\nb-3})
$$ solutions over $\F_p$. The number of solutions with any of the  additional restrictions $p\mid \nf_{\Ka/L}(\p)$,  or $p$ divides one of the expressions in (\ref{cong2}), or $\x$ is a singular point of $F(\x)=0$, is bounded by $O(p^{2\na+\nb-2})$. 
Hence, for $p\geq p_0$, there is a non-singular solution to $F(\x)=0$ over $\F_p$, such that $p\nmid \nf_{\Ka/L}(\p)$ and $p$ divides none of the expressions in (\ref{cong2}). This solution may be lifted via Hensel's lemma and shows that $\sig_p>0$ for $p\geq p_0$. 

We may assume without loss of generality that all primes $p<p_0$ are contained in $S_f\cup \{v_0\}$. For $p\mid M$, we recall that we have set $\v^{(M)}=(1,0,\ldots, 0)$. By Lemma \ref{lem:find-yw}, for any prime $p\in  S_f$, we are given a solution $(\y_p,\w_p)\in \cW(\Z_p)$ with $(\y_p,\w_p)\equiv (\y^{(M)},\w^{(M)})$ modulo $M$. Since $\cW$ is non-singular, this is already enough to establish $\sig_p>0$ for $p\in  S_f$. It remains to consider the prime $p_0$ corresponding to $v_0$ if this is a finite place. In the beginning of the argument we may assume that we are given some $p_0$-adic solution to $F(\x)=0$ with $\p=(1,0,\ldots,0)$. After renormalising we may even assume that this is an integral $p_0$-adic solution. Again since $W$ is smooth, this is sufficient to show that $\sig_{p_0}>0$, which thereby concludes the proof of \eqref{eq:pea}.

We now turn our attention to the singular integral \eqref{eq:thrush}. 
For $G$ sufficiently large, the domains $\calV$ and $\calW$ coincide with those defined in \cite{BHB}. Moreover, our singular integral $\sig_\infty$ is up to two fixed normalisation constants ($M^\na$ and $c$) the same as that defined in \cite{BHB}. Hence their arguments apply and show that 
$\sig_\infty\gg G^{1-2\na-\nb}H^{\na\nb}V_0^{2\na\nb}.$
The lower bounds for $\grS$ and $\sig_\infty$ in combination with (\ref{defN1}) and the fact that $\grS$ is absolutely convergent, now show that
\begin{equation}\label{eqnNdel2}
\calN_\xi^{(1)}(G,H,V)\gg G^{1-2\na-\nb}H^{\na\nb}V_0^{2\na\nb}.
\end{equation}
We now set $G=\log V$ and $ H=V^{\frac{1}{\na\nb^2(\na+16)}}$. 
Recalling \eqref{eq:UVW}, these choices clearly imply that $H^{\na\nb/2}\leq V$ and $H^{\na\nb/2-\na/2}<V_0^{\na}$; 
i.e. the assumptions of Lemma \ref{lem19} are satisfied. We choose $Q$ a fixed small power of $H$ such that $Q^6<H^{\na\nb/4}$ and (\ref{eqnQ1}) holds. 
Take $\xi$ sufficiently small. Then we may combine  \eqref{eqnNdel} with \eqref{eqn4.10b} and the lower bound  \eqref{eqnNdel2} to complete the proof of Theorem~\ref{thm:2}.

\end{document}